\documentclass{article}

\usepackage{amsfonts}
\usepackage{amsmath}
\usepackage{mathtools}
\usepackage{multirow}
\usepackage{epstopdf}
\usepackage{algorithm}
\usepackage{algorithmic}
\usepackage{bm}
\usepackage{microtype}
\usepackage{amssymb}
\usepackage{enumerate}
\usepackage{marvosym}
\usepackage{color}
\usepackage{marvosym}
\usepackage{subcaption}
\usepackage{hyperref}

\newtheorem{theorem}{Theorem}

\newtheorem{lemma}{Lemma}
\newtheorem{proof}{Proof}
\newtheorem{corollary}{Corollary}

\newcommand{\x}{\mathbf{x}}
\newcommand{\y}{\mathbf{y}}
\newcommand{\z}{\mathbf{z}}
\newcommand{\g}{\mathbf{g}}

\newcommand{\m}{\mathbf{m}}

\renewcommand{\v}{\mathbf{v}}

\newcommand{\bsigma}{\boldsymbol{\sigma}}

\newcommand{\E}{\mathbb{E}}
\newcommand{\bO}{{\cal O}}
\newcommand{\R}{\mathbb{R}}
\renewcommand{\P}{\mathcal{P}}
\newcommand{\F}{\mathcal{F}}

\newcommand{\<}{\left\langle}
\renewcommand{\>}{\right\rangle}

\usepackage[accepted]{icml2020}

\icmltitlerunning{Convergence Rate of RMSProp and Its Momentum Extension Measured by $\ell_1$ Norm}

\begin{document}

\onecolumn
\icmltitle{On the $\bO(\frac{\sqrt{d}}{T^{1/4}})$ Convergence Rate of RMSProp and Its Momentum Extension Measured by $\ell_1$ Norm}

\begin{icmlauthorlist}
\icmlauthor{Huan Li}{to}
\icmlauthor{Yiming Dong}{goo}
\icmlauthor{Zhouchen Lin}{goo}
\end{icmlauthorlist}

\icmlaffiliation{to}{Institute of Robotics and Automatic Information Systems, College of Artificial Intelligence, Nankai University, Tianjin, China.\\}
\icmlaffiliation{goo}{National Key Lab of General AI, School of Intelligence Science and Technology, Peking University, Beijing, China.\\}
\icmlcorrespondingauthor{Huan Li and Zhouchen Lin}{lihuanss@nankai.edu.cn, zlin@pku.edu.cn}





\vskip 0.3in



\printAffiliationsAndNotice{}  

\begin{abstract}
  Although adaptive gradient methods have been extensively used in deep learning, their convergence rates proved in the literature are all slower than that of SGD, particularly with respect to their dependence on the dimension. This paper considers the classical RMSProp and its momentum extension and establishes the convergence rate of $\frac{1}{T}\sum_{k=1}^T\E\left[\|\nabla f(\x^k)\|_1\right]\leq \bO(\frac{\sqrt{d}C}{T^{1/4}})$ measured by $\ell_1$ norm without the bounded gradient assumption, where $d$ is the dimension of the optimization variable, $T$ is the iteration number, and $C$ is a constant identical to that appeared in the optimal convergence rate of SGD. Our convergence rate matches the lower bound with respect to all the coefficients except the dimension $d$. Since $\|\x\|_2\ll\|\x\|_1\leq\sqrt{d}\|\x\|_2$ for problems with extremely large $d$, our convergence rate can be considered to be analogous to the $\frac{1}{T}\sum_{k=1}^T\E\left[\|\nabla f(\x^k)\|_2\right]\leq \bO(\frac{C}{T^{1/4}})$ rate of SGD in the ideal case of $\|\nabla f(\x)\|_1=\varTheta(\sqrt{d})\|\nabla f(\x)\|_2$.
\end{abstract}

\section{Introduction}\label{sec:intr}

This paper considers adaptive gradient methods for the following nonconvex smooth stochastic optimization problem:
\begin{eqnarray}
\begin{aligned}\label{problem}
\min_{\x\in\R^d} f(\x)=\E_{\xi\sim \P}[h(\x,\xi)],
\end{aligned}
\end{eqnarray}
where $\xi$ is a random variable and $\P$ is the data distribution.

When evaluating the convergence speed of an optimization method, traditional optimization theories primarily focus on the dependence on the iteration number. For example, it is well known that SGD reaches the precision of $\bO(\frac{1}{T^{1/4}})$ after $T$ iterations for nonconvex problem (\ref{problem}), disregarding the constants independent of $T$ within $\bO(\cdot)$. However, this measure is inadequate for high-dimensional applications, particularly in deep learning. Consider GPT-3, which possesses 175 billion parameters. In other words, in the training model (\ref{problem}),
\begin{eqnarray}
\begin{aligned}\notag
d=1.75\times 10^{11}\mbox{ in GPT-3}.
\end{aligned}
\end{eqnarray}
If a method converges with a rate of $\bO(\frac{d}{T^{1/4}})$, it is unrealistic since we rarely train a deep neural network for $10^{44}$ iterations. Therefore, it is desirable to study the explicit dependence on the dimension $d$ and the constants relying on $d$ in $\bO(\cdot)$, and furthermore, to decrease this dependence.

To compound the issue, although adaptive gradient methods, such as AdaGrad \citep{Duchi-2011-jmlr,McMahan-2010-colt}, RMSProp \citep{RMSProp-2012-hinton}, and Adam \citep{adam-15-iclr}, have become dominant in training deep neural networks, their convergence rates have not been thoroughly investigated, particularly with regard to their dependence on the dimension. Current analyses of convergence rates indicate that these methods often exhibit a strong dependence on the dimension. For example, recently, \citet{hong-2024-adagrad} (see Section \ref{lit-review} for the detailed literature reviews) proved the following state-of-the-art convergence rate for AdaGrad with high probability
\begin{eqnarray}
\begin{aligned}\label{lit-rate}
&\frac{1}{T}\sum_{k=1}^T\|\nabla f(\x^k)\|_2\leq \bO\left( \frac{\sqrt{d\ln T}}{T^{1/4}}\left(\sqrt[4]{\sigma_s^2L(f(\x^1)-f^*)}+\sigma_s\right) \right)
\end{aligned}
\end{eqnarray}
under assumption $\|\g^k-\nabla f(\x^k)\|^2\leq\sigma_s^2$, where $\g^k$ represents the stochastic gradient at $\x^k$. In contrast, the convergence rate of SGD \citep{Bottou-2018-siam} can be as fast as
\begin{eqnarray}
\begin{aligned}\label{SGD-rate}
\frac{1}{T}\sum_{k=1}^T\E\left[\|\nabla f(\x^k)\|_2\right]\leq \bO\left(\frac{\sqrt[4]{\sigma_s^2L(f(\x^1)-f^*)}}{T^{1/4}}\right)
\end{aligned}
\end{eqnarray}
with weaker assumption $\E\left[\|\g^k-\nabla f(\x^k)\|^2\right]\leq\sigma_s^2$. We observe that the convergence rate of SGD is $\sqrt{d\ln T}$ times faster than (\ref{lit-rate}). It remains an open problem of how to establish the convergence rate of adaptive gradient methods in a manner analogous to that of SGD, in order to bridge the gap between their rapid convergence observed in practice and their theoretically slower convergence rate compared to SGD.

\noindent\begin{minipage}[t]{0.46\textwidth}
\begin{algorithm}[H]
   \caption{RMSProp}
   \label{rmsprop}
\begin{algorithmic}
   \STATE Initialize $\x^1$, $\v_i^0$
   \FOR{$k=1,2,\cdots,T$}
   \STATE $\v^k=\beta\v^{k-1}+(1-\beta)(\g^k)^{\odot2}$
   \STATE $\x^{k+1}=\x^k-\frac{\eta}{\sqrt{\v^k}}\odot\g^k$
   \ENDFOR
   \STATE
\end{algorithmic}
\end{algorithm}
\end{minipage}
\begin{minipage}[t]{0.52\textwidth}
\begin{algorithm}[H]
   \caption{RMSProp with Momentum}
   \label{momentum}
\begin{algorithmic}
   \STATE Initialize $\x^1$, $\m_i^0=0$, $\v_i^0$
   \FOR{$k=1,2,\cdots,T$}
   \STATE $\v^k=\beta\v^{k-1}+(1-\beta)(\g^k)^{\odot2}$
   \STATE $\m^k=\theta\m^{k-1}+(1-\theta)\frac{1}{\sqrt{\v^k}}\odot\g^k$
   \STATE $\x^{k+1}=\x^k-\eta \m^k$
   \ENDFOR
\end{algorithmic}
\end{algorithm}
\end{minipage}

\subsection{Contribution}
In this paper, we consider the classical RMSProp and its momentum extension \citep{RMSProp-2012-hinton}, which are presented in Algorithms \ref{rmsprop} and \ref{momentum}, respectively. Specifically, for both methods, we prove the convergence rate of
\begin{eqnarray}
\begin{aligned}\notag
\frac{1}{T}\sum_{k=1}^T\E\left[\|\nabla f(\x^k)\|_1\right]\leq\widetilde\bO\left(\frac{\sqrt{d}}{T^{1/4}}\left(\sqrt[4]{\sigma_s^2L(f(\x^1)-f^*)}\right) + \frac{\sqrt{d}}{\sqrt{T}}\left(\sqrt{L(f(\x^1)-f^*)}\right) \right)
\end{aligned}
\end{eqnarray}
measured by $\ell_1$ norm under the assumption of coordinate-wise bounded noise variance, which does not require the boundedness of the gradient or stochastic gradient. Our convergence rate matches the lower bound established in \citep{Arjevani-2023-mp} with respect to $T$, $L$, $f(\x^1)-f^*$, and $\sigma_s$. Since $L(f(\x^1)-f^*)\geq \frac{1}{2}\|\nabla f(\x^1)\|^2=\frac{1}{2}\sum_{i=1}^d |\nabla_i f(\x^1)|^2$ and $\sigma_s^2\geq \sum_{i=1}^d\E\left[|\g_i^k-\nabla_i f(\x^k)|^2\right]$, they could assume large values in high-dimensional settings\footnote{However, empirical observations in deep learning training indicate that each element of $\nabla f(\x)$ tends to be very small, with both $\|\nabla f(\x^1)\|$ and $f(\x^1)-f^*$ typically being of order $\bO(1)$.}. So it is significant to achieve the optimal dependence on $L(f(\x^1)-f^*)$ and $\sigma_s$. The only coefficient left unclear whether it is tight measured by $\ell_1$ norm is the dimension $d$.

Note that $\|\x\|_2\ll\|\x\|_1\leq\sqrt{d}\|\x\|_2$ for any $\x\in\R^d$ with extremely large $d$, and additionally, $\|\x\|_1=\varTheta(\sqrt{d})\|\x\|_2$ when $\x$ is generated from uniform or Gaussian distribution. Therefore, our convergence rate can be considered to be analogous to (\ref{SGD-rate}) of SGD in the ideal case of $\|\nabla f(\x)\|_1=\varTheta(\sqrt{d})\|\nabla f(\x)\|_2$. Fortunately, as demonstrated in Figure \ref{figure1}, we have empirically observed that in real deep neural networks, the relationship $\|\nabla f(\x)\|_1=\varTheta(\sqrt{d})\|\nabla f(\x)\|_2$ holds true.

\begin{figure}[tbhp]
    \centering
    \includegraphics[width=\textwidth]{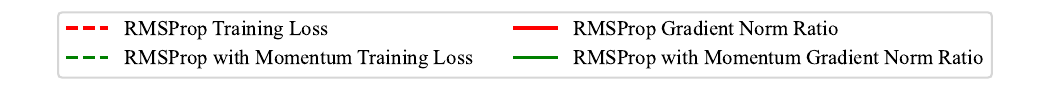}

    \begin{subfigure}{0.49\textwidth}
        \centering
        \includegraphics[width=\linewidth]{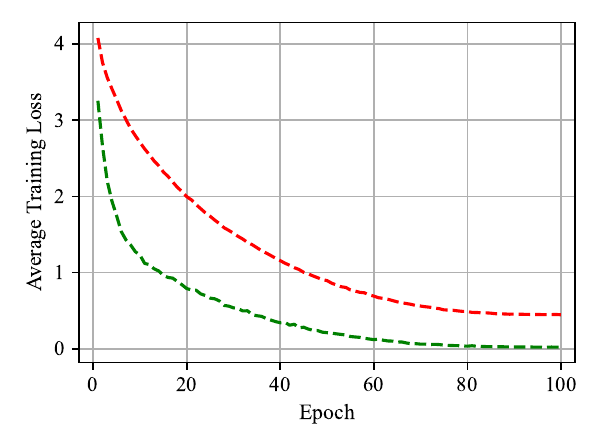}
        \caption{ResNet-50 training loss on Cifar-100}
    \end{subfigure}%
    \begin{subfigure}{0.49\textwidth}
        \centering
        \includegraphics[width=\linewidth]{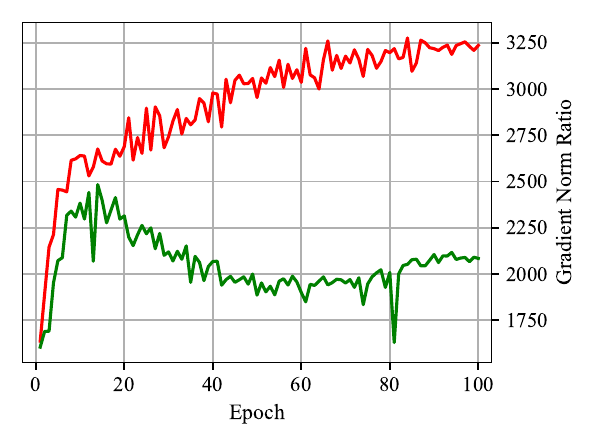}
        \caption{ResNet-50 gradient norm ratio ($\sqrt{d}\approx4869$)}
    \end{subfigure}

    \begin{subfigure}{0.49\textwidth}
        \centering
        \includegraphics[width=\linewidth]{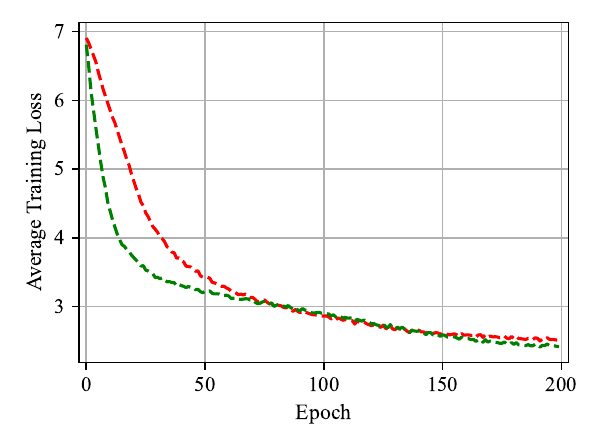}
        \caption{ResNet-50 training loss on ImageNet}
    \end{subfigure}%
    \begin{subfigure}{0.49\textwidth}
        \centering
        \includegraphics[width=\linewidth]{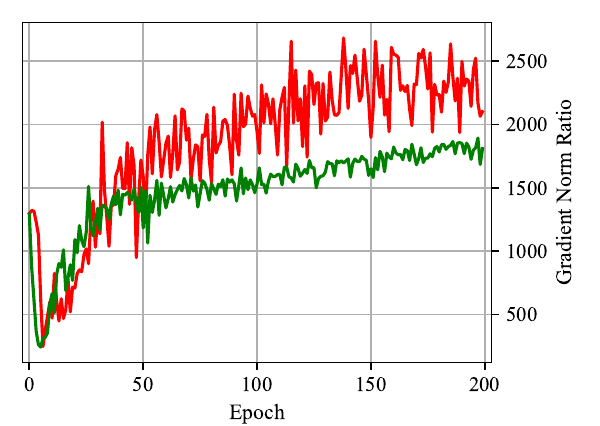}
        \caption{ResNet-50 gradient norm ratio ($\sqrt{d}\approx 5055$)}
    \end{subfigure}

    \begin{subfigure}{0.49\textwidth}
        \centering
        \includegraphics[width=\linewidth]{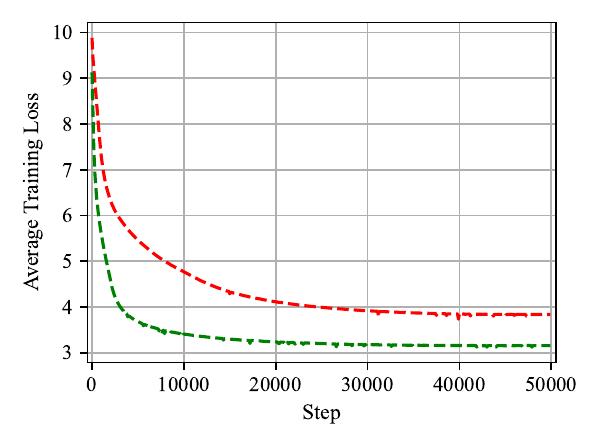}
        \caption{GPT-2 training loss on OpenWebText}
    \end{subfigure}%
    \begin{subfigure}{0.49\textwidth}
        \centering
        \includegraphics[width=\linewidth]{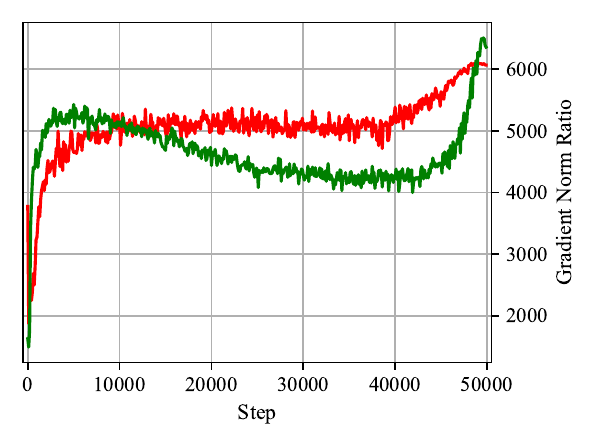}
        \caption{GPT-2 gradient norm ratio ($\sqrt{d}\approx  11148$)}
    \end{subfigure}
    \caption{Illustration of the relationship $\|\nabla f(\x^k)\|_1=\varTheta(\sqrt{d})\|\nabla f(\x^k)\|_2$. We use RMSProp and RMSProp with momentum to train ResNet50 on CIFAR-100 and ImageNet, and train GPT2 on the OpenWebText dataset. The gradient norm ratio shows $\frac{\|\nabla f(\x^k)\|_1}{\|\nabla f(\x^k)\|_2}$ and the average training loss shows the average loss over training samples.}
    \label{figure1}
\end{figure}

\subsection{Notations and Assumptions}\label{sec-notation}

Denote $\x_i$ and $\nabla_i f(\x)$ as the $i$th element of vectors $\x$ and $\nabla f(\x)$, respectively. Let $\x^k$ represent the value at iteration $k$. For scalars, such as $v$, we use $v_k$ instead of $v^k$ to denote its value at iteration $k$, while the latter represents its $k$th power. Denote $\|\cdot\|$, or $\|\cdot\|_2$ if emphasis is required, as the $\ell_2$ Euclidean norm and $\|\cdot\|_1$ as the $\ell_1$ norm for vectors, respectively. Denote $f^*=\inf f(\x)$. Denote $\odot$ to stand for the Hadamard product between vectors. Denote $\F_k=\sigma(\g^1,\g^2,\cdots,\g^k)$ to be the sigma field of the stochastic gradients up to $k$. Let $\E_{\F_k}[\cdot]$ denote the expectation with respect to $\F_k$ and $\E_k[\cdot|\F_{k-1}]$ the conditional expectation with respect to $\g^k$ conditioned on $\F_{k-1}$. We use $f=\bO(g)$, $f=\varOmega(g)$, and $f=\varTheta(g)$ to denote $f\leq c_1 g$, $f\geq c_2g$, and $c_2g\leq f\leq c_1g$ for some constants $c_1$ and $c_2$, respectively, and $\widetilde\bO$ to hide polylogarithmic factors. The base of natural logarithms is denoted by $e$.

Throughout this paper, we make the following assumptions:
\begin{enumerate}
\item Smoothness: $\|\nabla f(\y)-\nabla f(\x)\|\leq L\|\y-\x\|$,
\item Unbiased estimator: $\E_k\left[\g_i^k\big|\F_{k-1}\right]=\nabla_i f(\x^k)$,
\item Coordinate-wise bounded noise variance: $\E_k\left[|\g_i^k-\nabla_i f(\x^k)|^2\big|\F_{k-1}\right]\leq \sigma_i^2$.
\end{enumerate}
Denoting $\bsigma=[\sigma_1,\cdots,\sigma_d]$ and $\sigma_s=\|\bsigma\|_2=\sqrt{\sum_i \sigma_i^2}$, we have the standard bounded noise variance assumption
\begin{equation}\notag
\E_k\left[\|\g^k-\nabla f(\x^k)\|^2\big|\F_{k-1}\right]\leq\sigma_s^2,
\end{equation}
which is used in the analysis of SGD. Note that we do not assume the boundedness of $\nabla f(\x^k)$ or $\g^k$. 

\section{Convergence Rates of RMSProp and Its Momentum Extension}

In this section, we prove the convergence rates of the classical RMSProp and its momentum extension. Both methods are implemented in PyTorch by the following API with broad applications in deep learning:

\hspace*{3cm}\texttt{torch.optim.RMSprop(lr,...,momentum,...),}

\noindent where \texttt{momentum} and $\texttt{lr}$ equal $\theta$ and $(1-\theta)\eta$ in Algorithm \ref{momentum}, respectively. Specially, If we set \texttt{momentum=0}  in default, it reduces to RMSProp.

We establish the convergence rate of RMSProp with momentum in the following theorem. Additionally, if we set $\theta=0$, Theorem \ref{theorem} also provides the convergence rate of RMSProp. For brevity, we omit the details.
\begin{theorem}\label{theorem}
Suppose that Assumptions 1-3 hold. Let $\eta=\frac{\gamma}{\sqrt{dT}}$, $\beta=1-\frac{1}{T}$, $\v_i^0=\lambda\max\left\{\sigma_i^2,\frac{1}{dT}\right\},\forall i$, and $T\geq \frac{e^2}{\lambda}$, where $\theta\in[0,1)$, $\lambda\leq 1$, and $\gamma$ can be any constants serving as hyper-parameters for tuning performance in
practice. Then for Algorithm \ref{momentum}, we have
\begin{eqnarray}
\begin{aligned}\notag
\frac{1}{T}\sum_{k=1}^T\E\left[\|\nabla f(\x^k)\|_1\right]\leq \frac{d^{1/4}}{T^{1/4}}\sqrt{\frac{2F\|\bsigma\|_1}{\gamma}}+\frac{\sqrt{d}}{\sqrt{T}}\frac{4F}{\gamma},
\end{aligned}
\end{eqnarray}
where
\begin{eqnarray}
\begin{aligned}\label{def-F}
\frac{F}{\gamma}=\max\Bigg\{&1,\hspace*{0.3cm}3(2L\gamma+3)\ln(2L\gamma+3),\hspace*{0.3cm}\frac{3(f(\x^1)-f^*)}{\gamma},\\
&3\left(\frac{6e\sigma_s}{\sqrt{\lambda T}}+\frac{3L\gamma}{(1-\theta)^{1.5}}+3\right)\ln\left(\frac{6e\sigma_s}{\sqrt{\lambda T}}+\frac{3L\gamma}{(1-\theta)^{1.5}}+3\right),\\
&\left(\frac{6e\sigma_s}{\sqrt{\lambda T}}+\frac{3L\gamma}{(1-\theta)^{1.5}}\right)\ln\left(\frac{4L\gamma e^2}{\lambda\max\{d\min_i\sigma_i^2,\frac{1}{T}\}}\left(1+\frac{\theta^2}{2T(1-\theta)^2}\right)+\frac{12}{\lambda}\right)
\Bigg\}.
\end{aligned}
\end{eqnarray}
\end{theorem}

Since $\theta$ is a constant independent of $T$ and $d$, we can simplify Theorem \ref{theorem} in the following corollary.
\begin{corollary}\label{corollary1}
Under the settings of Theorem \ref{theorem}, letting $\gamma=\sqrt{\frac{f(\x^1)-f^*}{L}}$ and $T\geq\frac{\sigma_s^2}{\lambda L(f(\x^1)-f^*)}$, we have $\frac{F}{\gamma}=\widetilde\bO\left(\sqrt{L(f(\x^1)-f^*)}\right)$ and
\begin{eqnarray}
\begin{aligned}\notag
\frac{1}{T}\sum_{k=1}^T\E\left[\|\nabla f(\x^k)\|_1\right]\leq&\widetilde\bO\left(\frac{d^{1/4}}{T^{1/4}}\left(\sqrt[4]{\|\bsigma\|_1^2L(f(\x^1)-f^*)}\right) + \frac{\sqrt{d}}{\sqrt{T}}\left(\sqrt{L(f(\x^1)-f^*)}\right) \right)\\
\leq&\widetilde\bO\left(\frac{\sqrt{d}}{T^{1/4}}\left(\sqrt[4]{\sigma_s^2L(f(\x^1)-f^*)}\right) + \frac{\sqrt{d}}{\sqrt{T}}\left(\sqrt{L(f(\x^1)-f^*)}\right) \right),
\end{aligned}
\end{eqnarray}
where $\widetilde\bO$ hides $\ln(L(f(\x^1)-f^*))$ and $\ln\left(\frac{\sqrt{L(f(\x^1)-f^*)}}{\lambda\max\{d\min_i\sigma_i^2,\frac{1}{T}\}}+\frac{1}{\lambda}\right)$. If $T\geq \frac{dL(f(\x^1)-f^*)}{\|\bsigma\|_1^2}$, the first term dominates and the convergence rate becomes
\begin{eqnarray}
\begin{aligned}\label{our-rate1}
\frac{1}{T}\sum_{k=1}^T\E\left[\|\nabla f(\x^k)\|_1\right]\leq&\widetilde\bO\left(\frac{d^{1/4}}{T^{1/4}}\left(\sqrt[4]{\|\bsigma\|_1^2L(f(\x^1)-f^*)}\right)\right)\\
\leq&\widetilde\bO\left(\frac{\sqrt{d}}{T^{1/4}}\left(\sqrt[4]{\sigma_s^2L(f(\x^1)-f^*)}\right)\right).
\end{aligned}
\end{eqnarray}
On the other hand, if the noise variance is small, that is, $\|\bsigma\|_1^2\leq\frac{dL(f(\x^1)-f^*)}{T}$, the second term dominates and the convergence rate becomes
\begin{eqnarray}
\begin{aligned}\notag
\frac{1}{T}\sum_{k=1}^T\E\left[\|\nabla f(\x^k)\|_1\right]\leq\widetilde\bO\left( \frac{\sqrt{d}}{\sqrt{T}}\sqrt{L(f(\x^1)-f^*)} \right).
\end{aligned}
\end{eqnarray}
In deep learning with extremely large $d$, it can be expected that $d\min_i\sigma_i^2>\frac{1}{T}$, making it unlikely for $\ln T$ to appear in $\widetilde\bO$.
\end{corollary}

\textbf{Tightness with respect to the coefficients}. \citet{Arjevani-2023-mp} established the lower bound of stochastic optimization methods under the assumptions of smoothness and bounded noise variance. The convergence rate of SGD in (\ref{SGD-rate}) matches this lower bound. By comparing our convergence rate (\ref{our-rate1}) with (\ref{SGD-rate}), we observe that our convergence rate is also tight up to logarithmic factors with respect to the smoothness coefficient $L$, the initial function value gap $f(\x^1)-f^*$, the noise variance $\sigma_s$, and the iteration number $T$. The only coefficient left unclear whether it is tight measured by $\ell_1$ norm is the dimension $d$.

\textbf{Comparison to SGD}. Our convergence rate (\ref{our-rate1}) can be considered to be analogous to (\ref{SGD-rate}) of SGD in the ideal case of $\|\nabla f(\x)\|_1=\varTheta(\sqrt{d})\|\nabla f(\x)\|_2$. Fortunately, we have empirically observed that the relationship $\|\nabla f(\x)\|_1=\varTheta(\sqrt{d})\|\nabla f(\x)\|_2$ holds true in common deep neural networks, as shown in Figures \ref{figure1}. On the other hand, our theory relies on slightly stronger assumption of coordinate-wise bounded noise variance compared to SGD. Consequently, the constant $\sigma_s=\sqrt{\sum_i \sigma_i^2}$ is larger than the one used in SGD. Nonetheless, if each $\E_k\left[\|\g_i^k-\nabla_i f(\x^k)\|^2|\F_{k-1}\right]$ does not oscillate intensely during iterations, we may expect the two constants not to differ greatly.

\textbf{Two key points in our proof}. To establish the tight dependence on $L(f(\x^1)-f^*)$, we should upper bound $\sum_{i=1}^d\sum_{k=1}^T \E_{\F_{k-1}}\left[\sqrt{\widetilde\v_i^k}\right]$ (see the definition in (\ref{wide-v-define})) by $\sigma_s$ (or $\|\bsigma\|_1$) predominantly instead of $L(f(\x^1)-f^*)$. To address this issue, we provide a simple proof in Lemma \ref{lemma6} to bound $\sum_{i=1}^d\sum_{k=1}^T \E_{\F_{k-1}}\left[\sqrt{\widetilde\v_i^k}\right]$ by $\bO\left(T\|\bsigma\|_1+\frac{F}{\gamma}\sqrt{dT}\right)$, with the first term dominating. Additionally, to ensure the tight dependence on $\sigma_s$, we give a sharper upper bound for the error term in Lemma \ref{lemma3}, such that $\frac{F}{\gamma}$ in (\ref{def-F}) includes $\frac{\sigma_s}{\sqrt{\lambda T}}$ instead of just $\sigma_s$, where the former can be relaxed as $\frac{\sigma_s}{\sqrt{\lambda T}}\leq\sqrt{L(f(\x^1)-f^*)}$ by setting $T\geq\frac{\sigma_s^2}{\lambda L(f(\x^1)-f^*)}$.

\textbf{$\ell_1$ norm or $\ell_2$ norm}. The choice to utilize the $\ell_1$ norm is motivated by SignSGD \citep{Bernstein-2018-icml}, which is closely related to Adam \citep{Balles-2018-icml}. Technically, if we were to use the $\ell_2$ norm as conventionally done, we would need to make the right hand side of the convergence rate criteria independent of $d$ while remaining the other coefficients tight, as shown in (\ref{SGD-rate}), to make it no slower than SGD. However, achieving this target presents a challenge with current techniques. Instead, by using the $\ell_1$ norm, we can maintain the term $\sqrt{d}$ on the right hand side, as demonstrated in (\ref{our-rate1}), since $\|\nabla f(\x)\|_1=\varTheta(\sqrt{d})\|\nabla f(\x)\|_2$ in the ideal case.

\textbf{Relation to AdaGrad}. From the parameter settings in Corollary \ref{corollary1}, we have $\eta=\frac{\gamma}{\sqrt{dT}}$, $1-\beta=\frac{1}{T}$, $\v_i^0\geq\lambda\sigma_i^2\geq \varOmega(\frac{\sigma_i^2}{T})$, and $\frac{1}{e^2}\leq \beta^t\leq 1$ for any $t\leq T$ from (\ref{l3-equ3}). So for the update direction of $\x_i^k$, we have
\begin{eqnarray}
\begin{aligned}\notag
\eta\frac{\g_i^k}{\sqrt{\v_i^k}}=\eta\frac{\g_i^k}{\sqrt{ \beta^k\v_i^0 + (1-\beta)\sum_{t=1}^k\beta^{k-t}|\g_i^t|^2 }}\approx\frac{\gamma}{\sqrt{dT}}\frac{\g_i^k}{\sqrt{ \v_i^0+\frac{\sum_{t=1}^k|\g_i^t|^2}{T} }}\approx\frac{\gamma}{\sqrt{d}}\frac{\g_i^k}{\sqrt{\sigma_i^2+\sum_{t=1}^k|\g_i^t|^2 }}.
\end{aligned}
\end{eqnarray}
On the other hand, the update direction of $\x_i^k$ in AdaGrad is $\eta\frac{\g_i^k}{\sqrt{ \sum_{t=1}^k|\g_i^t|^2 }+\varepsilon}$. So in our parameter settings, RMSProp can be regarded as a refined variant of AdaGrad.


\section{Literature Comparisons}\label{lit-review}

It is not easy to compare convergence rates in the literature due to variations in assumptions. Furthermore, most literature does not state explicit dependence on the dimension in their theorems, and instead hides it within the proofs. In this section, we attempt to compare our convergence rate with the representative ones in the literature. Particularly, we primarily compare with the ones without the bounded gradient assumption. 

\subsection{Convergence Rate of AdaGrad in \citep{hong-2024-adagrad}}
\citet[Corollay 1]{hong-2024-adagrad} studied AdaGrad under the relaxed noise assumption of $\|\g^k-\nabla f(\x^k)\|^2\leq A(f(\x^k)-f^*)+B\|\nabla f(\x^k)\|^2+C$ with probability 1, and their result can be extended to the sub-Gaussian assumption where $\E\left[\exp(\frac{\|\g^k-\nabla f(\x^k)\|^2}{A(f(\x^k)-f^*)+B\|\nabla f(\x^k)\|^2+C})\right]\leq e$. We compare with their convergence rate by setting $C=\sigma_s^2$ and $A=B=0$. They proved the following convergence rate with high probability,
\begin{eqnarray}
\begin{aligned}\notag
&\frac{1}{T}\sum_{k=1}^T\|\nabla f(\x^k)\|_2^2\leq\bO\left(\triangle_1\left( \frac{\triangle_1+\sqrt{L\eta\triangle_1}}{T} + \frac{\sigma_s}{\sqrt{T}} \right)\right)=\bO\left(\frac{\sigma_s\triangle_1}{\sqrt{T}} \right),\\
&\mbox{where }\triangle_1=\bO\left(\frac{f(\x^1)-f^*}{\eta}+d\sigma_s \ln T + L\eta d^2\ln^2T\right).
\end{aligned}
\end{eqnarray}
$\triangle_1$ is minimized to be $\widetilde\bO\left(\left(\sqrt{L(f(\x^1)-f^*)}+\sigma_s\right)d\ln T\right)$ by letting $\eta=\sqrt{\frac{f(\x^1)-f^*}{L}}\frac{1}{d\ln T}$. Accordingly, their convergence rate is
\begin{eqnarray}
\begin{aligned}\notag
&\frac{1}{T}\sum_{k=1}^T\|\nabla f(\x^k)\|_2\leq\bO\left( \frac{\sqrt{d\ln T}}{T^{1/4}}\left(\sqrt[4]{\sigma_s^2L(f(\x^1)-f^*)}+\sigma_s\right) \right),
\end{aligned}
\end{eqnarray}
which is $\sqrt{d\ln T}$ times slower than (\ref{SGD-rate}) of SGD. It is also inferior to our convergence rate (\ref{our-rate1}) due to $\|\x\|_2\ll\|\x\|_1\leq\sqrt{d}\|\x\|_2$. Additionally, their dependence on $\sigma_s$ is not optimal, especially when $\sigma_s\geq\sqrt{L(f(\x^1)-f^*)}$. Hong and Lin also studied Adam in \citep{hong-2024-adam,JunhongLin-2023}, but the convergence rate is not superior to that of Adagrad. It should be noted that \citet{hong-2024-adagrad} established this result based on the assumption that $\|\g^k-\nabla f(\x^k)\|^2\leq \sigma_s^2$ with probability 1, or the sub-Gaussian assumption of $\E\left[\exp(\frac{\|\g^k-\nabla f(\x^k)\|^2}{\sigma_s^2})\right]\leq e$. In contrast, our assumption is $\E_k\left[\|\g_i^k-\nabla_i f(\x^k)\|^2\big|\F_{k-1}\right]\leq\sigma_i^2$ in the coordinate-wise manner. Determining which assumption is stronger is difficult.

\subsection{Convergence Rate of AdaGrad in \citep{Liu-2023-icml}}
\citet[Theorem 4.6]{Liu-2023-icml} studied AdaGrad under the coordinate-wise sub-Gaussian assumption of $\E\hspace*{-0.06cm}\left[\exp(\lambda^2|\g_i\hspace*{-0.06cm}-\hspace*{-0.06cm}\nabla_i f(\x)|^2)\right]\hspace*{-0.12cm}\leq\exp(\lambda^2\sigma_i^2),\forall|\lambda|\leq\frac{1}{\sigma_i} $. \citet{Liu-2023-icml} also used $\ell_1$ norm to measure the convergence rate. Specifically, from Theorem 4.6 and the corresponding proof on page 42 in \citep{Liu-2023-icml}, they proved the following convergence rate with probability at least $1-\delta$,
\begin{eqnarray}
\begin{aligned}\notag
&\frac{1}{T}\sum_{k=1}^T\|\nabla f(\x^k)\|_1^2\leq g(\delta)\bO\left( \frac{\|\bsigma\|_1}{\sqrt{T}}+\frac{r(\delta)}{T} \right),\\
&\mbox{where }g(\delta)=\bO\left( \frac{f(\x^1)\hspace*{-0.06cm}-\hspace*{-0.06cm}f^*}{\eta} \hspace*{-0.06cm}+\hspace*{-0.06cm} \left(d\|\bsigma\|_{\infty}\hspace*{-0.06cm}+\hspace*{-0.06cm}\sum_{i=1}^dc_i(\delta)\right)\sqrt{\log\frac{dT}{\delta}} \hspace*{-0.06cm}+\hspace*{-0.06cm} dL\eta\log \left(\|\bsigma\|_1\sqrt{T}\hspace*{-0.06cm}+\hspace*{-0.06cm}r(\delta)\right) \right),\\ &\hspace*{1.1cm}c_i(\delta)=\bO\left(\sigma_i^3\log\frac{d}{\delta}+\sigma_i\log\left(1+\sigma_i^2T+\sigma_i^2\log\frac{d}{\delta}\right)+\|\bsigma\|_1\log(\|\bsigma\|_1\sqrt{T}+r(\delta))\right),\\
&\hspace*{1.1cm}r(\delta)=\bO\left(f(\x^1)-f^*+\|\bsigma^2\|_1\log\frac{d}{\delta}+\|\bsigma\|_1\sqrt{\log\frac{d}{\delta}}+Ld\log L\right).
\end{aligned}
\end{eqnarray}
It is not easy to simplify the above convergence rate and the dependence on $\bsigma$ is not optimal due to the second term in $g(\delta)$. Ignoring $\sum_{i=1}^dc_i(\delta)$ and the logarithmic term in $g(\delta)$, their $\frac{1}{T}\sum_{k=1}^T\|\nabla f(\x^k)\|_1$ is upper bounded by a constant not less than $\frac{\sqrt{d\|\bsigma\|_{\infty}\|\bsigma\|_1}+\sqrt[4]{\|\bsigma\|_1^2dL(f(\x^1)-f^*)}}{T^{1/4}}$, which is inferior to our convergence rate (\ref{our-rate1}).

\subsection{Convergence Rate of RMSProp in \citep{luo-2020-iclr}}
\citet[Theorem 4.3]{luo-2020-iclr} studied problem $\min_{\x} f(\x)=\sum_{j=0}^{n-1} f_j(\x)$ and they assumed $\sum_{j=0}^{n-1}\|\nabla f_j(\x)\|_2^2\leq D_1\|\nabla f(\x)\|_2^2+D_0$. They did not give the explicit dependence on the dimension in their theorem 4.3 and we try to recover it from their proof. On page 37 in \citep{luo-2020-iclr}, the authors gave
\begin{eqnarray}
\begin{aligned}\notag
&\min_{k\in[t_{init},T]}\min\left\{\|\nabla f(\x^{k,0})\|_1,\|\nabla f(\x^{k,0})\|_2^2\sqrt{\frac{D_1d}{D_0}}\right\}\leq \frac{Q_{1,3}+Q_{2,3}\log T}{\sqrt{T}-\sqrt{t_{init}-1}} +\sqrt{D_0}Q_{3,3},\\
&\mbox{where } Q_{1,3}=\frac{f(\x^{t_{init,0}})-f^*-C_6\log t_{init}}{2\eta_1}\sqrt{10dnD_1}d,\quad Q_{2,3}=\frac{C_6\sqrt{10dnD_1}d}{2\eta_1},\\
&\hspace*{1.1cm} C_6=L\eta_1^2\left(\frac{nd}{2(1-\beta_2)}+\frac{C_4\sqrt{d}}{n\sqrt{1-\beta_2}}\right),\quad C_4\geq\frac{dn^2}{(1-\beta_2)^{1.5}},\quad 1-\beta_2\leq\bO\left(\frac{1}{n^{3.5}}\right),
\end{aligned}
\end{eqnarray}
and the other notations can be found in their proof. Since
\begin{eqnarray}
\begin{aligned}\notag
&Q_{1,3}+Q_{2,3}\log T\geq \varOmega\left(\left(\frac{f(\x^{t_{init,0}})-f^*}{\eta_1}+\frac{L\eta_1nd^{1.5}}{(1-\beta_2)^2}(\log T-\log t_{init})\right)\sqrt{dnD_1}d\right)\\
&\geq \varOmega\left(\sqrt{D_1 L(f(\x^{t_{init,0}})-f^*)}d^{9/4}\right).
\end{aligned}
\end{eqnarray}
We see that their $\min_k\|\nabla f(\x^{k,0})\|_2^2$ is upper bounded by a constant not less than $\widetilde\varTheta\Big(\frac{D_0Q_{3,3}}{\sqrt{dD_1}}+\frac{d^{7/4}\sqrt{D_0 L(f(\x^{t_{init,0}})-f^*)}}{\sqrt{T}}\Big)$. That is, $\min_k\|\nabla f(\x^{k,0})\|_2$ is upper bounded by a constant not less than $\widetilde\bO\Big(\frac{\sqrt{D_0Q_{3,3}}}{\sqrt[4]{dD_1}}+\frac{d^{7/8}\sqrt[4]{D_0 L(f(\x^{t_{init,0}})-f^*)}}{T^{1/4}}\Big)$, which is at least $d^{3/8}$ times slower than our convergence rate (\ref{our-rate1}).

\subsection{Convergence Rate of RMSProp in \citep{bottou-2022-tmlr}}

\citet[Theorem 2]{bottou-2022-tmlr} studied the convergence rate of RMSProp under the bounded stochastic gradient assumption, that is, there is $R\geq\sqrt{\varepsilon}$ so that $\|\g^k\|_{\infty}\leq R-\sqrt{\varepsilon}$ almost surely. They proved the following bound for RMSProp
\begin{eqnarray}
\begin{aligned}\notag
\E\left[\|\nabla f(\x^{\tau})\|_2^2\right]\leq \bO\left( R\frac{f(\x^1)-f^*}{\eta T} + \left(\frac{dR^2}{\sqrt{1-\beta}}+\frac{\eta d RL}{1-\beta}\right)\left(\frac{1}{T}\ln\left(1+\frac{R^2}{(1-\beta)\varepsilon}\right)-\ln\beta\right) \right).
\end{aligned}
\end{eqnarray}
Letting $\beta=1-\frac{1}{T}$ and $\eta=\frac{1}{\sqrt{dT}}\sqrt{\frac{f(\x^1)-f^*}{L}}$, the above convergence rate can be simplified to 
\begin{eqnarray}
\begin{aligned}\notag
\E\left[\|\nabla f(\x^{\tau})\|_2^2\right]\leq \bO\left(\left(\frac{\sqrt{d}}{\sqrt{T}}R\sqrt{L(f(\x^1)-f^*)} + \frac{dR^2}{\sqrt{T}}\right)\ln(RT)\right),
\end{aligned}
\end{eqnarray}
where $dR^2\geq \|\g^k\|_2^2,\forall k$. Due to the bounded stochastic gradient assumption, comparing their convergence rate with ours is unfair. Our proof follows the analytical framework in \citep{bottou-2022-tmlr}. However, unlike their work, we cannot lower bound $\frac{|\nabla_i f(\x^k)|^2}{\sqrt{\widetilde \v_i^k}}$ by $\frac{|\nabla_i f(\x^k)|^2}{R}$ without the bounded stochastic gradient assumption. Instead, we use 
\begin{eqnarray}
\begin{aligned}\notag
\left(\sum_{k=1}^K\E\left[\|\nabla f(\x^k)\|_1\right]\right)^2\leq \left(\sum_{k=1}^K\sum_{i=1}^d\E\left[\frac{|\nabla_i f(\x^k)|^2}{\sqrt{\widetilde \v_i^k}}\right]\right)\left(\sum_{k=1}^K\sum_{i=1}^d\E\left[\sqrt{\widetilde\v_i^k}\right]\right)
\end{aligned}
\end{eqnarray}
and rigorously derive a tight upper bound for $\sum_{k=1}^K\sum_{i=1}^d\E\left[\sqrt{\widetilde\v_i^k}\right]$. Furthermore, in the absence of the bounded stochastic gradient assumption, we cannot use $\v_i^k\leq R^2$ to upper bound $\sum_{k=1}^K\E\left[\frac{|\g_i^k|^2}{\v_i^k}\right]$. To address this, we employ mathematical induction to establish the bound in (\ref{equ11}). Additionally, we provide a sharper upper bound for the error term in Lemma \ref{lemma3} to ensure the tight dependence on $\sigma_s$.

\subsection{Convergence Rate of Adam in \citep{haochuanli-2023}}
\citet[Theorem 4.1]{haochuanli-2023} introduced a new proof of boundedness of gradients along the optimization trajectory. Although their Theorem 4.1 has no explicit dependence on $d$, it has a higher dependence on $L(f(\x^1)-f^*)$, $\sigma_s$ and the constant $\lambda$ as a compromise, where $\lambda$ appears in the adaptive step-size $\frac{\eta}{\sqrt{\v^k}+\lambda}$, which is usually small in practice, for example, $\lambda=10^{-8}$ in PyTorch implementation. Specifically, they assumed $\|\g^k-\nabla f(\x^k)\|\leq \sigma_s$ with probability 1 and proved $\frac{1}{T}\sum_{k=1}^T\|\nabla f(\x^k)\|_2^2\leq\epsilon^2$ with high probability by letting
\begin{eqnarray}
\begin{aligned}\notag
&T=\max\left\{\frac{1}{\beta^2},\frac{G(f(\x^1)-f^*)}{\eta\epsilon^2}\right\},\quad \eta\leq\min\left\{\frac{\sigma_s\lambda\beta}{LG},\frac{\lambda^{3/2}\beta}{L\sqrt{G}}\right\},\quad \beta\leq\bO\left( \frac{\lambda\epsilon^2}{\sigma_s^2G}\right),
\end{aligned}
\end{eqnarray}
and $G$ to be a large constant satisfying $G\geq \max\{\lambda,\sigma_s,\sqrt{L(f(\x^1)-f^*)}\}$. From their setting, we see that $T\geq \frac{G^{2.5} \sigma_s^2L(f(\x^1)-f^*)}{\lambda^{2.5}\epsilon^4}$. Consequently, their $\frac{1}{T}\sum_{k=1}^T\|\nabla f(\x^k)\|_2$ is upper bounded by a constant not less than $(\frac{G}{\lambda})^{5/8}\frac{\sqrt[4]{\sigma_s^2L(f(\x^1)-f^*)}}{T^{1/4}}$, which is at least $(\frac{G}{\lambda})^{5/8}$ times slower than SGD. It is not easy to compare with our convergence rate (\ref{our-rate1}) due to different measurement. When $\|\nabla f(\x)\|_1\geq\varOmega((\frac{\lambda}{G})^{5/8}\sqrt{d})\|\nabla f(\x)\|_2$, our convergence rate is superior, and in the ideal case of $\|\nabla f(\x)\|_1=\varTheta(\sqrt{d})\|\nabla f(\x)\|_2$, our convergence rate is also $(\frac{G}{\lambda})^{5/8}$ times faster.

\subsection{Other works}

There are other literature that analyze adaptive gradient methods, including \citep{Bottou-2020-jmlr,Kavis-2022-iclr,Faw-2022-colt,WeiChen-2023-colt,Attia-2023-icml} for AdaGrad-norm, \citep{WeiChen-2023-colt} for AdaGrad, \citep{Zou-2019-cvpr,bottou-2022-tmlr} for RMSProp, \citep{Reddi-2018-iclr,Zou-2019-cvpr,bottou-2022-tmlr,tianbaoyang-2021,Chen-2022-jmlr,luo-2022-nips,Wang-2023-nips,JunhongLin-2023,hong-2024-adam,Qi-Zhang-2024-rmsporp} for Adam, and \citep{Zaheer-2018-nips,Loshchilov-2018-iclr,chen-2019-iclr,luo-2019-iclr,You-2019-iclr,Zhuang-2020-nips,Chen-2021-ijcai,Savarese-2021-cvpr,Crawshaw-2022,xie-2022} for other variants. However, none have established a convergence rate comparable to that of SGD.

\section{Proof of Theorem \ref{theorem}}
Denote $\x^0=\x^1$, which corresponds to $\m^0=0$. The second and third steps of Algorithm \ref{momentum} can be rewritten in the heavy-ball style equivalently as follows,
\begin{eqnarray}
\begin{aligned}\label{x-update}
\x^{k+1}=\x^k-\frac{\eta(1-\theta)}{\sqrt{\v^k}}\odot \g^k+\theta(\x^k-\x^{k-1}),\forall k\geq 1,
\end{aligned}
\end{eqnarray}
which leads to
\begin{eqnarray}
\begin{aligned}\notag
\x^{k+1}-\theta\x^k=\x^k-\theta\x^{k-1}-\frac{\eta(1-\theta)}{\sqrt{\v^k}}\odot \g^k.
\end{aligned}
\end{eqnarray}
We follow \citep{Yin-2020-nips} to define
\begin{eqnarray}
\begin{aligned}\label{z-define}
\z^k=\frac{1}{1-\theta}\x^k-\frac{\theta}{1-\theta}\x^{k-1},\forall k\geq 1.
\end{aligned}
\end{eqnarray}
Specially, we have $\z^1=\x^1$ since $\x^1=\x^0$. Thus, we have
\begin{eqnarray}
\begin{aligned}\label{z-update}
\z^{k+1}=\z^k-\frac{\eta}{\sqrt{\v^k}}\odot \g^k,
\end{aligned}
\end{eqnarray}
and
\begin{eqnarray}
\begin{aligned}\label{diff-z-x}
\z^k-\x^k=\frac{\theta}{1-\theta}(\x^k-\x^{k-1}).
\end{aligned}
\end{eqnarray}
We follow \citep{bottou-2022-tmlr,Faw-2022-colt} to define
\begin{equation}\label{wide-v-define}
\widetilde \v_i^k=\beta \v_i^{k-1}+(1-\beta)\left(|\nabla_i f(\x^k)|^2+\sigma_i^2\right).
\end{equation}
Then with the supporting lemmas in Section \ref{sec-supporting-lemmas} we can prove Theorem \ref{theorem}.

\begin{proof}
As the gradient is $L$-Lipschitz, we have
\begin{eqnarray}
\begin{aligned}\notag
\E_k\left[f(\z^{k+1})\big|\F_{k-1}\right]-f(\z^k)\leq& \E_k\left[\<\nabla f(\z^k),\z^{k+1}-\z^k\>+\frac{L}{2}\|\z^{k+1}-\z^k\|^2\Big|\F_{k-1}\right]\\
=& \E_k\left[-\eta\sum_{i=1}^d\<\nabla_i f(\z^k),\frac{\g_i^k}{\sqrt{\v_i^k}}\>+\frac{L\eta^2}{2}\sum_{i=1}^d\frac{|\g_i^k|^2}{\v_i^k}\Big|\F_{k-1}\right],
\end{aligned}
\end{eqnarray}
where we use (\ref{z-update}). Decomposing the first term into
\begin{eqnarray}
\begin{aligned}\notag
&-\<\nabla_i f(\x^k),\frac{\g_i^k}{\sqrt{\widetilde\v_i^k}}\>+\<\nabla_i f(\x^k),\frac{\g_i^k}{\sqrt{\widetilde\v_i^k}}-\frac{\g_i^k}{\sqrt{\v_i^k}}\>+\<\nabla_i f(\x^k)-\nabla_i f(\z^k),\frac{\g_i^k}{\sqrt{\v_i^k}}\>
\end{aligned}
\end{eqnarray}
and using Assumption 2, we have
\begin{eqnarray}
\begin{aligned}\label{equ0}
\E_k\left[f(\z^{k+1})\big|\F_{k-1}\right]-f(\z^k)\leq& -\eta\sum_{i=1}^d\frac{|\nabla_i f(\x^k)|^2}{\sqrt{\widetilde\v_i^k}}+\frac{L\eta^2}{2}\sum_{i=1}^d\E_k\left[\frac{|\g_i^k|^2}{\v_i^k}\Big|\F_{k-1}\right]\\
&+\underbrace{\eta\sum_{i=1}^d\E_k\left[\<\nabla_i f(\x^k),\frac{\g_i^k}{\sqrt{\widetilde\v_i^k}}-\frac{\g_i^k}{\sqrt{\v_i^k}}\>\Big|\F_{k-1}\right]}_{\text{\rm term (a)}}\\
&+\underbrace{\eta\sum_{i=1}^d\E_k\left[\<\nabla_i f(\x^k)-\nabla_i f(\z^k),\frac{\g_i^k}{\sqrt{\v_i^k}}\>\Big|\F_{k-1}\right]}_{\text{\rm term (b)}}.
\end{aligned}
\end{eqnarray}
We can use Lemma \ref{lemma3} to bound term (a). For term (b), we have
\begin{eqnarray}
\hspace*{-2cm}\begin{aligned}\notag
&\eta\sum_{i=1}^d\<\nabla_i f(\x^k)-\nabla_i f(\z^k),\frac{\g_i^k}{\sqrt{\v_i^k}}\>\\
\leq&\frac{(1-\theta)^{0.5}}{2L\theta^{0.5}}\sum_{i=1}^d|\nabla_i f(\x^k)-\nabla_i f(\z^k)|^2+\frac{L\theta^{0.5}\eta^2}{2(1-\theta)^{0.5}}\sum_{i=1}^d\frac{|\g_i^k|^2}{\v_i^k}\\
=&\frac{(1-\theta)^{0.5}}{2L\theta^{0.5}}\|\nabla f(\x^k)-\nabla f(\z^k)\|^2+\frac{L\theta^{0.5}\eta^2}{2(1-\theta)^{0.5}}\sum_{i=1}^d\frac{|\g_i^k|^2}{\v_i^k}\\
\overset{(1)}\leq&\frac{L\theta^{1.5}\eta^2}{2(1-\theta)^{0.5}}\sum_{t=1}^{k-1}\theta^{k-1-t}\sum_{i=1}^d\frac{|\g_i^t|^2}{\v_i^t}+\frac{L\theta^{0.5}\eta^2}{2(1-\theta)^{0.5}}\sum_{i=1}^d\frac{|\g_i^k|^2}{\v_i^k}\\
=&\frac{L\sqrt{\theta}\eta^2}{2\sqrt{1-\theta}}\left(\sum_{t=1}^{k-1}\theta^{k-t}\sum_{i=1}^d\frac{|\g_i^t|^2}{\v_i^t}+\sum_{i=1}^d\frac{|\g_i^k|^2}{\v_i^k}\right)=\frac{L\sqrt{\theta}\eta^2}{2\sqrt{1-\theta}}\sum_{t=1}^k\theta^{k-t}\sum_{i=1}^d\frac{|\g_i^t|^2}{\v_i^t},
\end{aligned}
\end{eqnarray}
where we use Lemma \ref{lemma2} in $\overset{(1)}\leq$. Plugging the above inequality and Lemma \ref{lemma3} into (\ref{equ0}), taking expectation on $\F_{k-1}$, and rearranging the terms, we have
\begin{eqnarray}
\begin{aligned}\notag
&\E_{\F_k}\left[f(\z^{k+1})\right]-\E_{\F_{k-1}}\left[f(\z^k)\right]+\frac{\eta}{2}\sum_{i=1}^d\E_{\F_{k-1}}\left[\frac{|\nabla_i f(\x^k)|^2}{\sqrt{\widetilde\v_i^k}}\right]\\
\leq& \frac{2\eta e(1-\beta)}{\sqrt{\lambda}}\sum_{i=1}^d\sigma_i\E_{\F_k}\left[\frac{|\g_i^k|^2}{\v_i^k}\right]+\frac{L\eta^2}{\sqrt{1-\theta}}\sum_{t=1}^{k}\theta^{k-t}\sum_{i=1}^d\E_{\F_t}\left[\frac{|\g_i^t|^2}{\v_i^t}\right].
\end{aligned}
\end{eqnarray}
Summing over $k=1,\cdots,K$, we have
\begin{eqnarray}
\begin{aligned}\label{equ5-1}
&\E_{\F_K}\left[f(\z^{K+1})\right]-f^*+\frac{\eta}{2}\sum_{i=1}^d\sum_{k=1}^K\E_{\F_{k-1}}\left[\frac{|\nabla_i f(\x^k)|^2}{\sqrt{\widetilde\v_i^k}}\right]\hspace*{1cm}\\
\leq& f(\z^1)-f^*+\frac{2\eta e(1-\beta)}{\sqrt{\lambda}}\sum_{i=1}^d\sigma_i\sum_{k=1}^K\E_{\F_k}\left[\frac{|\g_i^k|^2}{\v_i^k}\right]+\frac{L\eta^2}{\sqrt{1-\theta}}\sum_{i=1}^d\sum_{k=1}^K\sum_{t=1}^{k}\theta^{k-t}\E_{\F_t}\left[\frac{|\g_i^t|^2}{\v_i^t}\right]\\
\leq& f(\z^1)-f^*+\frac{2\eta e(1-\beta)}{\sqrt{\lambda}}\underbrace{\sum_{i=1}^d\sigma_i\sum_{k=1}^K\E_{\F_k}\left[\frac{|\g_i^k|^2}{\v_i^k}\right]}_{\text{\rm term (c)}}+\frac{L\eta^2}{(1-\theta)^{1.5}}\underbrace{\sum_{i=1}^d\sum_{t=1}^K\E_{\F_t}\left[\frac{|\g_i^t|^2}{\v_i^t}\right]}_{\text{\rm term (d)}},
\end{aligned}
\end{eqnarray}
where we use
\begin{eqnarray}
\begin{aligned}\notag
&\sum_{i=1}^d\sum_{k=1}^K\sum_{t=1}^{k}\theta^{k-t}\E_{\F_t}\left[\frac{|\g_i^t|^2}{\v_i^t}\right]=\sum_{i=1}^d\sum_{t=1}^K\sum_{k=t}^{K}\theta^{k-t}\E_{\F_t}\left[\frac{|\g_i^t|^2}{\v_i^t}\right]\leq \frac{1}{1-\theta}\sum_{i=1}^d\sum_{t=1}^K\E_{\F_t}\left[\frac{|\g_i^t|^2}{\v_i^t}\right].
\end{aligned}
\end{eqnarray}
Using Lemmas \ref{lemma5} and \ref{lemma4} to bound terms (c) and (d), respectively, letting $\eta=\frac{\gamma}{\sqrt{dT}}$ and $\beta=1-\frac{1}{T}$, we have
\begin{eqnarray}
\begin{aligned}\label{equ8}
&\E_{\F_K}\left[f(\z^{K+1})\right]-f^*+\frac{\eta}{2}\sum_{i=1}^d\sum_{k=1}^K\E_{\F_{k-1}}\left[\frac{|\nabla_i f(\x^k)|^2}{\sqrt{\widetilde\v_i^k}}\right]\\
\leq& f(\z^1)-f^*+\hspace*{-0.07cm}\left(\hspace*{-0.07cm}\frac{2\eta e\sqrt{d}\sigma_s}{\sqrt{\lambda}}+\frac{L\eta^2d}{(1-\theta)^{1.5}(1-\beta)}\right)\hspace*{-0.08cm}\ln\Bigg(\frac{4Le^2(1-\beta)}{\lambda\max\{d\min_i\sigma_i^2,\frac{1}{T}\}}\E_{\F_K}\hspace*{-0.15cm}\left[ \sum_{k=1}^K(f(\z^k)-f^*)\right.\hspace*{-0.2cm}\\
&\hspace*{6cm}\left.+\frac{L\theta^2 \eta^2}{2(1-\theta)^2}\sum_{t=1}^{K-1}\sum_{i=1}^d\frac{|\g_i^t|^2}{\v_i^t} \right]+\frac{e^2}{\lambda}+e+1\Bigg)\\
=& f(\z^1)-f^*+\left(\frac{2e\gamma\sigma_s}{\sqrt{\lambda T}}+\frac{L\gamma^2}{(1-\theta)^{1.5}}\right)\ln\Bigg(\frac{4Le^2(1-\beta)}{\lambda\max\{d\min_i\sigma_i^2,\frac{1}{T}\}}\E_{\F_K}\left[ \sum_{k=1}^K(f(\z^k)-f^*)\right.\\
&\hspace*{6cm}\left.+\frac{L\theta^2 \gamma^2}{2(1-\theta)^2}\frac{1-\beta}{d}\sum_{t=1}^{K-1}\sum_{i=1}^d\frac{|\g_i^t|^2}{\v_i^t} \right]+\frac{e^2}{\lambda}+e+1\Bigg).
\end{aligned}
\end{eqnarray}
Next, we bound the right hand side of (\ref{equ8}) by the constant $F$ defined in (\ref{def-F}). Specifically, we will prove
\begin{eqnarray}
\begin{aligned}\label{equ11}
\E_{\F_{k-1}}\left[f(\z^k)\right]-f^*\leq F\mbox{ and }\frac{1-\beta}{d}\sum_{i=1}^d\sum_{t=1}^{k-1}\E_{\F_{k-1}}\left[\frac{|\g_i^t|^2}{\v_i^t}\right]\leq \frac{F}{L\gamma^2}
\end{aligned}
\end{eqnarray}
by induction. (\ref{equ11}) holds for $k=1$ from the definition of $F$ in (\ref{def-F}), $\x^1=\z^1$, and $\sum_{i=1}^d\sum_{t=1}^{k-1}\frac{|\g_i^t|^2}{\v_i^t}=0$ when $k=1$ given in Lemma \ref{lemma2}. Suppose that the two inequalities hold for all $k=1,2,\cdots,K$. Now, we consider $k=K+1$. From Lemma \ref{lemma4}, we have
\begin{eqnarray}
\begin{aligned}\label{equ14}
\frac{1-\beta}{d}\sum_{i=1}^d\sum_{t=1}^K\E_{\F_K}\left[\frac{|\g_i^t|^2}{\v_i^t}\right]\leq& \ln\Bigg(\frac{4Le^2(1-\beta)}{\lambda\max\{d\min_i\sigma_i^2,\frac{1}{T}\}}\E_{\F_K}\left[\sum_{k=1}^K(f(\z^k)-f^*)\right.\\
&\hspace*{0.7cm}\left.+\frac{L\theta^2 \gamma^2}{2(1-\theta)^2}\frac{1-\beta}{d}\sum_{i=1}^d\sum_{t=1}^{K-1}\frac{|\g_i^t|^2}{\v_i^t}\right]+\frac{e^2}{\lambda}+e+1\Bigg)\\
\leq& \ln\left(\frac{4Le^2(1-\beta)}{\lambda\max\{d\min_i\sigma_i^2,\frac{1}{T}\}}\left( KF+\frac{\theta^2}{2(1-\theta)^2}F\right)+\frac{e^2}{\lambda}+ e+1\right)\\
\leq& \ln\left(\frac{4L\gamma e^2}{\lambda\max\{d\min_i\sigma_i^2,\frac{1}{T}\}}\left(1+\frac{\theta^2}{2T(1-\theta)^2}\right)\frac{F}{\gamma}+\frac{12}{\lambda}\frac{F}{\gamma}\right)\\
\overset{(2)}\leq& \frac{F}{L\gamma^2},
\end{aligned}
\end{eqnarray}
where we use $(1-\beta)K=\frac{K}{T}\leq 1$ and let $\frac{F}{\gamma}\geq 1$. Inequality $\overset{(2)}\leq$ will be verified later. Using the similar proof to (\ref{equ14}), we derive from (\ref{equ8}) that
\begin{eqnarray}
\begin{aligned}\label{equ9}
&\E_{\F_K}\left[f(\z^{K+1})\right]-f^*+\frac{\eta}{2}\sum_{i=1}^d\sum_{k=1}^K\E_{\F_{k-1}}\left[\frac{|\nabla_i f(\x^k)|^2}{\sqrt{\widetilde\v_i^k}}\right]\\
\leq& f(\z^1)-f^*+\left(\frac{2e\gamma\sigma_s}{\sqrt{\lambda T}}+\frac{L\gamma^2}{(1-\theta)^{1.5}}\right)\ln\left(\frac{4L\gamma e^2}{\lambda\max\{d\min_i\sigma_i^2,\frac{1}{T}\}}\left(1+\frac{\theta^2}{2T(1-\theta)^2}\right)\frac{F}{\gamma}+\frac{12}{\lambda}\frac{F}{\gamma}\right)\\
\overset{(3)}\leq& F.
\end{aligned}
\end{eqnarray}
We construct $F$ for $\overset{(2)}\leq$ and $\overset{(3)}\leq$ to hold by letting
\begin{eqnarray}
\begin{aligned}\notag
&1\leq\frac{F}{\gamma},\hspace*{2cm} \ln\left(\frac{4L\gamma e^2}{\lambda\max\{d\min_i\sigma_i^2,\frac{1}{T}\}}\left(1+\frac{\theta^2}{2T(1-\theta)^2}\right)+\frac{12}{\lambda}\right)\leq \frac{F}{2L\gamma^2},\\
&\ln\frac{F}{\gamma}\leq \frac{F}{2L\gamma^2},\hspace*{1.3cm} f(\z^1)-f^*\leq\frac{F}{3},\hspace*{1.2cm}\left(\frac{2e\gamma\sigma_s}{\sqrt{\lambda T}}+\frac{L\gamma^2}{(1-\theta)^{1.5}}\right)\ln\frac{F}{\gamma}\leq \frac{F}{3},\\
&\left(\frac{2e\gamma\sigma_s}{\sqrt{\lambda T}}+\frac{L\gamma^2}{(1-\theta)^{1.5}}\right)\ln\left(\frac{4L\gamma e^2}{\lambda\max\{d\min_i\sigma_i^2,\frac{1}{T}\}}\left(1+\frac{\theta^2}{2T(1-\theta)^2}\right)+\frac{12}{\lambda}\right)\leq \frac{F}{3},
\end{aligned}
\end{eqnarray}
which are satisfied by setting of $F$ in (\ref{def-F}), where we use $\x^1=\z^1$ and $c\ln x\leq x$ for all $x\geq 3c\ln c$ and $c\geq 3$ proved in Appendix \ref{appendixB}. So (\ref{equ11}) also holds for $k=K+1$. Thus, (\ref{equ11}) holds for all $k=1,2,\cdots,T$ by induction.

Using Holder's inequality, Lemma \ref{lemma6}, and (\ref{equ9}), we have
\begin{eqnarray}
\begin{aligned}\notag
&\left(\sum_{k=1}^K\E_{\F_{k-1}}\left[\|\nabla f(\x^k)\|_1\right]\right)^2=\left(\sum_{k=1}^K\sum_{i=1}^d\E_{\F_{k-1}}\left[|\nabla_i f(\x^k)|\right]\right)^2\\
\leq& \left(\sum_{k=1}^K\sum_{i=1}^d\E_{\F_{k-1}}\left[\frac{|\nabla_i f(\x^k)|^2}{\sqrt{\widetilde\v_i^k}}\right]\right)\left(\sum_{k=1}^K\sum_{i=1}^d\E_{\F_{k-1}}\left[\sqrt{ \widetilde\v_i^k}\right]\right)\\
\leq& \left(\sum_{k=1}^K\sum_{i=1}^d\E_{\F_{k-1}}\left[\frac{|\nabla_i f(\x^k)|^2}{\sqrt{\widetilde\v_i^k}}\right]\right)\left( K\|\bsigma\|_1+\sqrt{dT}+2\sum_{k=1}^K\sum_{i=1}^d\E_{\F_{k-1}}\left[\frac{|\nabla_i f(\x^k)|^2}{\sqrt{\widetilde\v_i^k}}\right] \right)\\
\leq& \frac{2F}{\eta}\left( K\|\bsigma\|_1+\frac{F}{\eta} + \frac{4F}{\eta} \right)= \frac{2F}{\eta}\left( K\|\bsigma\|_1 + \frac{5F}{\eta} \right)
\end{aligned}
\end{eqnarray}
for all $K\leq T$, where we use $\eta=\frac{\gamma}{\sqrt{dT}}$ and $\frac{F}{\gamma}\geq 1$. So we have
\begin{eqnarray}
\begin{aligned}\notag
\frac{1}{T}\sum_{k=1}^T\E_{\F_{k-1}}\left[\|\nabla f(\x^k)\|_1\right]\leq \frac{1}{T}\left(\sqrt{\frac{2FT\|\bsigma\|_1}{\eta}} + \frac{4F}{\eta}\right)= \frac{d^{1/4}}{T^{1/4}}\sqrt{\frac{2F\|\bsigma\|_1}{\gamma}}+\frac{\sqrt{d}}{\sqrt{T}}\frac{4F}{\gamma}.
\end{aligned}
\end{eqnarray}
\end{proof}

\subsection{Supporting Lemmas}\label{sec-supporting-lemmas}
In this section, we give some technical lemmas that will be used in our analysis.
\begin{lemma}\label{lemma1}\citep{bottou-2022-tmlr}
Let $v_t=\beta v_{t-1}+(1-\beta)g_t^2$. Then we have
\begin{eqnarray}
\begin{aligned}\notag
&(1-\beta)\sum_{t=1}^k\frac{g_t^2}{v_t}\leq \ln\frac{v_k}{\beta^kv_0}.
\end{aligned}
\end{eqnarray}
\end{lemma}

The next lemma is motivated by \citep{bottou-2022-tmlr}. However, the key distinction is that following the proof in \citep{bottou-2022-tmlr}, we can only get $\sigma_i\sqrt{1-\beta}\E_k\left[\frac{|\g_i^k|^2}{\v_i^k}\Big|\F_{k-1}\right]$ in the last component of (\ref{equ15}), where $1-\beta=\frac{1}{T}$. We strengthen the constant from $\sqrt{1-\beta}$ to $1-\beta$, which is crucial to achieve the tight dependence on $\sigma_s$ in our theory.
\begin{lemma}\label{lemma3}
Suppose that Assumption 3 holds. Define $\widetilde\v_i^k$ as in (\ref{wide-v-define}). Let $\v_i^0=\lambda\max\{\sigma_i^2,\frac{1}{dT}\}$, $\beta=1-\frac{1}{T}$, and $T\geq \frac{e^2}{\lambda}\geq 2$. Then we have
\begin{eqnarray}
\begin{aligned}\label{equ15}
&\E_k\left[\<\nabla_i f(\x^k),\frac{\g_i^k}{\sqrt{\widetilde\v_i^k}}-\frac{\g_i^k}{\sqrt{\v_i^k}}\>\Big|\F_{k-1}\right]\leq \frac{|\nabla_i f(\x^k)|^2}{2\sqrt{\widetilde\v_i^k}}+\frac{2\sigma_ie(1-\beta)}{\sqrt{\lambda}}\E_k\left[\frac{|\g_i^k|^2}{\v_i^k}\Big|\F_{k-1}\right].
\end{aligned}
\end{eqnarray}
\end{lemma}
\begin{proof}
From the definition of $\widetilde\v_i^k$, the recursion of $\v_i^k$, and the setting of $\v_i^0$, we have
\begin{eqnarray}
\begin{aligned}\label{l3-equ1}
&\widetilde\v_i^k\geq\beta\v_i^{k-1}=\beta\left(\beta^{k-1}\v_i^0+(1-\beta)\sum_{t=1}^{k-1}\beta^{k-1-t}|\g_i^t|^2\right)\geq \beta^k\v_i^0\geq \frac{\lambda\sigma_i^2}{e^2},
\end{aligned}
\end{eqnarray}
where we use $\beta^k\geq \frac{1}{e^2}$ for any $k\leq T$ from
\begin{eqnarray}
\begin{aligned}\label{l3-equ3}
k\ln\beta=-k\ln\frac{1}{\beta}\geq -T\frac{1-\beta}{\beta}=-T\frac{\frac{1}{T}}{1-\frac{1}{T}}\geq -2,
\end{aligned}
\end{eqnarray}
since $\ln x\leq x-1$ for any $x>0$.
So we have
\begin{eqnarray}
\begin{aligned}\notag
\left|\frac{1}{\sqrt{\widetilde\v_i^k}}-\frac{1}{\sqrt{\v_i^k}}\right|&=\frac{\left|\v_i^k-\widetilde\v_i^k\right|}{\sqrt{\widetilde\v_i^k}\sqrt{\v_i^k}\left(\sqrt{\v_i^k}+\sqrt{\widetilde\v_i^k}\right)}=(1-\beta)\frac{\left||\g_i^k|^2-|\nabla_i f(\x^k)|^2-\sigma_i^2\right|}{\sqrt{\widetilde\v_i^k}\sqrt{\v_i^k}\left(\sqrt{\v_i^k}+\sqrt{\widetilde\v_i^k}\right)}\\
&\leq (1-\beta)\frac{\left|\g_i^k-\nabla_i f(\x^k)\right|\left|\g_i^k+\nabla_i f(\x^k)\right|+\sigma_i^2}{\sqrt{\widetilde\v_i^k}\sqrt{\v_i^k}\left(\sqrt{\v_i^k}+\sqrt{\widetilde\v_i^k}\right)}\\
&\overset{(1)}\leq \sqrt{1-\beta}\frac{|\g_i^k-\nabla_i f(\x^k)|}{\sqrt{\widetilde\v_i^k}\sqrt{\v_i^k}} + \frac{e(1-\beta)}{\sqrt{\lambda}}\frac{\sigma_i}{\sqrt{\widetilde\v_i^k}\sqrt{\v_i^k}},
\end{aligned}
\end{eqnarray}
and
\begin{eqnarray}
\begin{aligned}\label{l3-equ2}
&\E_k\left[\<\nabla_i f(\x^k),\frac{\g_i^k}{\sqrt{\widetilde\v_i^k}}-\frac{\g_i^k}{\sqrt{\v_i^k}}\>\Big|\F_{k-1}\right]\\
\leq&\sqrt{1-\beta} \E_k\left[\frac{|\g_i^k-\nabla_i f(\x^k)||\nabla_i f(\x^k)||\g_i^k|}{\sqrt{\widetilde\v_i^k}\sqrt{\v_i^k}}\Big|\F_{k-1}\right]+\frac{\sigma_ie(1-\beta)}{\sqrt{\lambda}}\E_k\left[\frac{|\nabla_i f(\x^k)||\g_i^k|}{\sqrt{\widetilde\v_i^k}\sqrt{\v_i^k}}\Big|\F_{k-1}\right],
\end{aligned}
\end{eqnarray}
where we use the definitions of $\widetilde\v_i^k$ and $\v_i^k$ and (\ref{l3-equ1}) in $\overset{(1)}\leq$.
For the first term, we have
\begin{eqnarray}
\begin{aligned}\notag
&\sqrt{1-\beta}\E_k\left[\frac{|\g_i^k-\nabla_i f(\x^k)||\nabla_i f(\x^k)||\g_i^k|}{\sqrt{\widetilde\v_i^k}\sqrt{\v_i^k}}\Big|\F_{k-1}\right]\\
\leq&\frac{|\nabla_i f(\x^k)|^2}{4\sigma_i^2\sqrt{\widetilde\v_i^k}}\E_k\left[|\g_i^k-\nabla_i f(\x^k)|^2\big|\F_{k-1}\right]+\frac{\sigma_i^2(1-\beta)}{\sqrt{\widetilde\v_i^k}}\E_k\left[\frac{|\g_i^k|^2}{\v_i^k}\Big|\F_{k-1}\right]\\
\overset{(2)}\leq&\frac{|\nabla_i f(\x^k)|^2}{4\sqrt{\widetilde\v_i^k}}+\frac{\sigma_ie(1-\beta)}{\sqrt{\lambda}}\E_k\left[\frac{|\g_i^k|^2}{\v_i^k}\Big|\F_{k-1}\right],
\end{aligned}
\end{eqnarray}
where we use Assumption 3 and (\ref{l3-equ1}) in $\overset{(2)}\leq$. For the second term, we have
\begin{eqnarray}
\begin{aligned}\notag
\frac{\sigma_ie(1-\beta)}{\sqrt{\lambda}}\E_k\left[\frac{|\nabla_i f(\x^k)||\g_i^k|}{\sqrt{\widetilde\v_i^k}\sqrt{\v_i^k}}\Big|\F_{k-1}\right]\leq&\frac{|\nabla_i f(\x^k)|^2}{4\sqrt{\widetilde\v_i^k}}+\frac{\sigma_i^2e^2(1-\beta)^2}{\lambda\sqrt{\widetilde\v_i^k}}\E_k\left[\frac{|\g_i^k|^2}{\v_i^k}\Big|\F_{k-1}\right]\\
\overset{(3)}\leq&\frac{|\nabla_i f(\x^k)|^2}{4\sqrt{\widetilde\v_i^k}}+\frac{\sigma_ie^3(1-\beta)^2}{\lambda^{1.5}}\E_k\left[\frac{|\g_i^k|^2}{\v_i^k}\Big|\F_{k-1}\right]\\
\leq&\frac{|\nabla_i f(\x^k)|^2}{4\sqrt{\widetilde\v_i^k}}+\frac{\sigma_ie(1-\beta)}{\sqrt{\lambda}}\E_k\left[\frac{|\g_i^k|^2}{\v_i^k}\Big|\F_{k-1}\right],
\end{aligned}
\end{eqnarray}
where we use (\ref{l3-equ1}) again in $\overset{(3)}\leq$. Plugging the above two inequalities into (\ref{l3-equ2}), we have the conclusion.
\end{proof}

The next lemma is used to bound the norm of the gradient by the function value gap and the second order term, where the latter corresponds to $\frac{|\g_i^t|^2}{\v_i^t}$ from the second order term in Taylor expansion.
\begin{lemma}\label{lemma2}
Suppose that Assumption 1 holds. Letting $\m^0=0$, we have
\begin{eqnarray}
\begin{aligned}\notag
&\|\nabla f(\x^k)-\nabla f(\z^k)\|^2\leq \frac{L^2\theta^2\eta^2}{1-\theta}\sum_{t=1}^{k-1}\theta^{k-1-t}\sum_{i=1}^d\frac{|\g_i^t|^2}{\v_i^t},\\
&\|\nabla f(\x^k)\|^2\leq 4L(f(\z^k)-f^*)+\frac{2 L^2\theta^2\eta^2}{1-\theta}\sum_{t=1}^{k-1}\theta^{k-1-t}\sum_{i=1}^d\frac{|\g_i^t|^2}{\v_i^t},\\
&\sum_{k=1}^K\|\nabla f(\x^k)\|^2\leq 4L\left(\sum_{k=1}^K(f(\z^k)-f^*)+\frac{L\theta^2\eta^2}{2(1-\theta)^2}\sum_{t=1}^{K-1}\sum_{i=1}^d\frac{|\g_i^t|^2}{\v_i^t}\right).
\end{aligned}
\end{eqnarray}
Specially, denote $\sum_{t=1}^{K-1}\sum_{i=1}^d\frac{|\g_i^t|^2}{\v_i^t}=0$ when $K=1$.
\end{lemma}
\begin{proof}
For the first part, as the gradient is $L$-Lipschitz, we have
\begin{eqnarray}
\begin{aligned}\notag
\|\nabla f(\x^k)-\nabla f(\z^k)\|^2&\leq L^2\|\x^k-\z^k\|^2\overset{(1)}=\frac{L^2\theta^2}{(1-\theta)^2}\|\x^k-\x^{k-1}\|^2\overset{(2)}=\frac{L^2\theta^2\eta^2}{(1-\theta)^2}\|\m^{k-1}\|^2,
\end{aligned}
\end{eqnarray}
where we use (\ref{diff-z-x}) in $\overset{(1)}=$ and the update of $\x$ in $\overset{(2)}=$. From the update of $\m^k$ in Algorithm \ref{momentum}, we have
\begin{eqnarray}
\begin{aligned}\notag
&\m_i^k=\theta^k \m_i^0+(1-\theta)\sum_{t=1}^{k}\theta^{k-t} \frac{\g_i^t}{\sqrt{\v_i^t}}=(1-\theta)\sum_{t=1}^{k}\theta^{k-t} \frac{\g_i^t}{\sqrt{\v_i^t}}.
\end{aligned}
\end{eqnarray}
Using the convexity of $(\cdot)^2$, we have
\begin{eqnarray}
\begin{aligned}\notag
&|\m_i^k|^2\leq(1-\theta)^2\left(\sum_{t=1}^{k}\theta^{k-t}\right)\left(\sum_{t=1}^{k}\theta^{k-t} \frac{|\g_i^t|^2}{\v_i^t}\right)\leq(1-\theta)\sum_{t=1}^{k}\theta^{k-t} \frac{|\g_i^t|^2}{\v_i^t}.
\end{aligned}
\end{eqnarray}
For the second part, we have
\begin{eqnarray}
\begin{aligned}\notag
\|\nabla f(\x^k)\|^2\leq&2\|\nabla f(\z^k)\|^2+2\|\nabla f(\x^k)-\nabla f(\z^k)\|^2\\
\leq&4L(f(\z^k)-f^*)+\frac{2 L^2\theta^2\eta^2}{1-\theta}\sum_{t=1}^{k-1}\theta^{k-1-t}\sum_{i=1}^d\frac{|\g_i^t|^2}{\v_i^t},
\end{aligned}
\end{eqnarray}
where we use
\begin{eqnarray}
\begin{aligned}\notag
f^*&\leq f\left(\z^k-\frac{1}{L}\nabla f(\z^k)\right)\\
&\leq f(\z^k)-\frac{1}{L}\<\nabla f(\z^k),\nabla f(\z^k)\>+\frac{L}{2}\left\|\frac{1}{L}\nabla f(\z^k)\right\|^2\\
&=f(\z^k)-\frac{1}{2L}\|\nabla f(\z^k)\|^2.
\end{aligned}
\end{eqnarray}
For the third part, we have
\begin{eqnarray}
\begin{aligned}\notag
&\sum_{k=1}^K\|\nabla f(\x^k)\|^2\leq\sum_{k=1}^K\left(4L(f(\z^k)-f^*)+\frac{2 L^2\theta^2 \eta^2}{1-\theta}\sum_{t=1}^{k-1}\theta^{k-1-t}\sum_{i=1}^d\frac{|\g_i^t|^2}{\v_i^t}\right).
\end{aligned}
\end{eqnarray}
Using
\begin{eqnarray}
\begin{aligned}\notag
\sum_{k=1}^K\sum_{t=1}^{k-1}\theta^{k-1-t}\sum_{i=1}^d\frac{|\g_i^t|^2}{\v_i^t}=&\sum_{t=1}^{K-1}\sum_{k=t+1}^{K}\theta^{k-1-t}\sum_{i=1}^d\frac{|\g_i^t|^2}{\v_i^t}\leq\frac{1}{1-\theta}\sum_{t=1}^{K-1}\sum_{i=1}^d\frac{|\g_i^t|^2}{\v_i^t},
\end{aligned}
\end{eqnarray}
we have the conclusion. Specially, when $K=1$, we have $\sum_{k=1}^K\|\nabla f(\x^k)\|^2=\|\nabla f(\x^1)\|^2=\|\nabla f(\z^1)\|^2\leq 2L(f(\z^1)-f^*)$. So we can denote $\sum_{t=1}^{K-1}\sum_{i=1}^d\frac{|\g_i^t|^2}{\v_i^t}=0$ when $K=1$ such that the third part holds for all $K\geq 1$.
\end{proof}

The next two lemmas are used to bound the second order terms in (\ref{equ5-1}).
\begin{lemma}\label{lemma5}
Suppose that Assumptions 1-3 hold. Let $\v_i^0=\lambda\max\left\{\sigma_i^2,\frac{1}{dT}\right\}$ and $\beta=1-\frac{1}{T}$. Then for all $K\leq T$, we have
\begin{eqnarray}
\begin{aligned}\notag
&\frac{1-\beta}{\sqrt{d}}\sum_{i=1}^d\sigma_i\sum_{k=1}^K\E_{\F_K}\left[\frac{|\g_i^k|^2}{\v_i^k}\right]\\
\leq& \sigma_s\ln\left(\frac{4Le^2(1-\beta)}{\lambda\max\{d\min_i\sigma_i^2,\frac{1}{T}\}}\E_{\F_K}\left[ \sum_{k=1}^K(f(\z^k)-f^*)+\frac{L\theta^2 \eta^2}{2(1-\theta)^2}\sum_{t=1}^{K-1}\sum_{i=1}^d\frac{|\g_i^t|^2}{\v_i^t} \right]+\frac{e^2}{\lambda}+e+1\right).
\end{aligned}
\end{eqnarray}
\end{lemma}
\begin{proof}
From Lemma \ref{lemma1}, the concavity of $\ln x$, Holder's inequality, and the definition of $\sigma_s=\sqrt{\sum_i \sigma_i^2}$, we have
\begin{eqnarray}
\begin{aligned}\label{l5-equ1}
&\frac{1-\beta}{\sqrt{d}}\sum_{i=1}^d\sigma_i\sum_{k=1}^K\E_{\F_K}\left[\frac{|\g_i^k|^2}{\v_i^k}\right]\leq\frac{1}{\sqrt{d}}\sum_{i=1}^d\sigma_i\E_{\F_K}\left[\ln \frac{\v_i^K}{\beta^K\v_i^0}\right]\leq\frac{1}{\sqrt{d}}\sum_{i=1}^d\sigma_i\ln \E_{\F_K}\left[\frac{\v_i^K}{\beta^K\v_i^0}\right]\\
&\leq\sqrt{\frac{1}{d}\sum_{i=1}^d\sigma_i^2\sum_{i=1}^d\left(\ln\E_{\F_K}\left[\frac{\v_i^K}{\beta^K\v_i^0}\right]\right)^2}=\sigma_s\sqrt{\frac{1}{d}\sum_{i=1}^d\left(\ln\E_{\F_K}\left[\frac{\v_i^K}{\beta^K\v_i^0}\right]\right)^2}.
\end{aligned}
\end{eqnarray}
From the recursion of $\v^k$, we have $\v_i^K\geq \beta^K\v_i^0$, which leads to $\ln \E_{\F_K}\left[\frac{\v_i^K}{\beta^K\v_i^0}\right]\geq 0$. From the concavity of $(\ln x)^2$ for $x\geq e$, we have
\begin{eqnarray}
\begin{aligned}\label{l5-equ2}
\frac{1}{d}\sum_{i=1}^d\left(\ln\E_{\F_K}\left[\frac{\v_i^K}{\beta^K\v_i^0}\right]\right)^2\leq& \frac{1}{d}\sum_{i=1}^d\left(\ln\left(\E_{\F_K}\left[\frac{\v_i^K}{\beta^K\v_i^0}\right]+e\right)\right)^2\\
\leq& \left(\ln\left(\frac{1}{d}\sum_{i=1}^d\E_{\F_K}\left[\frac{\v_i^K}{\beta^K\v_i^0}\right]+e\right)\right)^2.
\end{aligned}
\end{eqnarray}
Using the recursive update of $\v^k$, Assumptions 2 and 3, (\ref{l3-equ3}), $\v_i^0=\lambda\max\left\{\sigma_i^2,\frac{1}{dT}\right\}$, and Lemma \ref{lemma2}, we have
\begin{eqnarray}
\begin{aligned}\notag
&\frac{1}{d}\sum_{i=1}^d\E_{\F_K}\left[\frac{\v_i^K}{\beta^K\v_i^0}\right]\\
=& \frac{1}{d}\sum_{i=1}^d\frac{(1-\beta)\sum_{k=1}^K\beta^{K-k}\E_{\F_K}\left[|\g_i^k|^2\right]+\beta^K \v_i^0}{\beta^K \v_i^0}\\
\leq& \frac{1}{d}\sum_{i=1}^d\frac{(1-\beta)\sum_{k=1}^K\beta^{K-k}\E_{\F_K}\left[|\nabla_i f(\x^k)|^2+\sigma_i^2\right]}{\beta^K \v_i^0}+1\\
\leq& \frac{e^2}{d}\sum_{i=1}^d\frac{(1-\beta)\sum_{k=1}^K\beta^{K-k}\E_{\F_K}\left[|\nabla_i f(\x^k)|^2\right]+\sigma_i^2}{ \lambda\max\left\{\sigma_i^2,\frac{1}{dT}\right\} }+1\\
\leq&\frac{e^2(1-\beta)}{\lambda\max\{d\min_i\sigma_i^2,\frac{1}{T}\}}\sum_{k=1}^K\beta^{K-k}\E_{\F_K}\left[\|\nabla f(\x^k)\|^2\right]+\frac{e^2}{\lambda}+1\\
\leq&\frac{4Le^2(1-\beta)}{\lambda\max\{d\min_i\sigma_i^2,\frac{1}{T}\}}\E_{\F_K}\left[ \sum_{k=1}^K(f(\z^k)-f^*)+\frac{L\theta^2 \eta^2}{2(1-\theta)^2}\sum_{t=1}^{K-1}\sum_{i=1}^d\frac{|\g_i^t|^2}{\v_i^t} \right]+\frac{e^2}{\lambda}+1.
\end{aligned}
\end{eqnarray}
Plugging into (\ref{l5-equ2}) and (\ref{l5-equ1}), we have the conclusion.
\end{proof}

Specially, replacing each $\sigma_i$ by 1 and $\sigma_s$ by $\sqrt{d}$ in the inductive derivation of (\ref{l5-equ1}), we have the following lemma.
\begin{lemma}\label{lemma4}
Suppose that Assumptions 1-3 hold. Let $\v_i^0=\lambda\max\left\{\sigma_i^2,\frac{1}{dT}\right\}$ and $\beta=1-\frac{1}{T}$. Then for all $K\leq T$, we have
\begin{eqnarray}
\begin{aligned}\notag
&\frac{1-\beta}{d}\sum_{i=1}^d\sum_{k=1}^K\E_{\F_K}\left[\frac{|\g_i^k|^2}{\v_i^k}\right]\\
\leq& \ln\left(\frac{4Le^2(1-\beta)}{\lambda\max\{d\min_i\sigma_i^2,\frac{1}{T}\}}\E_{\F_K}\left[ \sum_{k=1}^K(f(\z^k)-f^*)+\frac{L\theta^2 \eta^2}{2(1-\theta)^2}\sum_{t=1}^{K-1}\sum_{i=1}^d\frac{|\g_i^t|^2}{\v_i^t} \right]+\frac{e^2}{\lambda}+e+1\right).
\end{aligned}
\end{eqnarray}
\end{lemma}

In the next lemma, the key point is that the dominant part on the right hand side of (\ref{l6-equ1}) only depends on $\|\bsigma\|_1$ instead of $L(f(\x^1)-f^*)$, which is crucial to give tight dependence on $L(f(\x^1)-f^*)$ in our theoretical analysis. From (\ref{equ9}), the third part on the right hand side of (\ref{l6-equ1}) is in fact of order $\bO(\sqrt{dT})$, confirming that the first term indeed dominates the bound.
\begin{lemma}\label{lemma6}
Suppose that Assumptions 1-3 hold. Let $\beta\leq 1$ and $\v_i^0=\lambda\max\left\{\sigma_i^2,\frac{1}{dT}\right\}$ with $\lambda\leq 1$. Then for all $K\leq T$, we have
\begin{eqnarray}
\begin{aligned}\label{l6-equ1}
\sum_{i=1}^d\sum_{k=1}^K \E_{\F_{k-1}}\left[\sqrt{\widetilde\v_i^k}\right]\leq K\|\bsigma\|_1+\sqrt{dT}+2\sum_{t=1}^K\sum_{i=1}^d\E_{\F_{t-1}}\left[\frac{|\nabla_i f(\x^t)|^2}{\sqrt{\widetilde\v_i^t}}\right].
\end{aligned}
\end{eqnarray}
\end{lemma}
\begin{proof}
From the definition of $\widetilde\v_i^k$, we have
\begin{eqnarray}
\begin{aligned}\notag
&\E_{\F_{k-1}}\left[\sqrt{\widetilde\v_i^k}\right]\\
=&\E_{\F_{k-1}}\left[\sqrt{\beta\v_i^{k-1}+(1-\beta)\left(|\nabla_i f(\x^k)|^2+\sigma_i^2\right)}\right]\\
=&\E_{\F_{k-1}}\left[\frac{\beta\v_i^{k-1}+(1-\beta)\sigma_i^2}{\sqrt{\beta\v_i^{k-1}+(1-\beta)\left(|\nabla_i f(\x^k)|^2+\sigma_i^2\right)}} + \frac{(1-\beta)|\nabla_i f(\x^k)|^2}{\sqrt{\beta\v_i^{k-1}+(1-\beta)\left(|\nabla_i f(\x^k)|^2+\sigma_i^2\right)}}\right]\\
\leq&\E_{\F_{k-1}}\left[\sqrt{\beta\v_i^{k-1}+(1-\beta)\sigma_i^2}\right]+(1-\beta)\E_{\F_{k-1}}\left[\frac{|\nabla_i f(\x^k)|^2}{\sqrt{\widetilde\v_i^k}}\right].
\end{aligned}
\end{eqnarray}
Consider the first part in the general case. From the recursion of $\v_i^k$, we have
\begin{eqnarray}
\hspace*{-1cm}\begin{aligned}\notag
&\E_{\F_{k-t}}\left[\sqrt{\beta^t\v_i^{k-t}+(1-\beta^t)\sigma_i^2}\right]\\
=&\E_{\F_{k-t}}\left[\sqrt{\beta^{t+1}\v_i^{k-t-1}+\beta^t(1-\beta)|\g_i^{k-t}|^2+(1-\beta^t)\sigma_i^2}\right]\\
=&\E_{\F_{k-t-1}}\left[\E_{k-t}\left[\sqrt{\beta^{t+1}\v_i^{k-t-1}+\beta^t(1-\beta)|\g_i^{k-t}|^2+(1-\beta^t)\sigma_i^2}\Big|\F_{k-t-1}\right]\right]\\
\overset{(1)}\leq&\E_{\F_{k-t-1}}\left[\sqrt{\beta^{t+1}\v_i^{k-t-1}+\beta^t(1-\beta)\E_{k-t}\left[|\g_i^{k-t}|^2|\F_{k-t-1}\right]+(1-\beta^t)\sigma_i^2}\right]\\
\overset{(2)}\leq&\E_{\F_{k-t-1}}\left[\sqrt{\beta^{t+1}\v_i^{k-t-1}+\beta^t(1-\beta)\left(|\nabla_i f(\x^{k-t})|^2+\sigma_i^2\right)+(1-\beta^t)\sigma_i^2}\right]\\
=&\E_{\F_{k-t-1}}\left[\sqrt{\beta^{t+1}\v_i^{k-t-1}+\beta^t(1-\beta)|\nabla_i f(\x^{k-t})|^2+(1-\beta^{t+1})\sigma_i^2}\right]\\
=&\E_{\F_{k-t-1}}\left[\frac{\beta^{t+1}\v_i^{k-t-1}+(1-\beta^{t+1})\sigma_i^2}{\sqrt{\beta^{t+1}\v_i^{k-t-1}+\beta^t(1-\beta)|\nabla_i f(\x^{k-t})|^2+(1-\beta^{t+1})\sigma_i^2}}\right]\\
& + \E_{\F_{k-t-1}}\left[\frac{\beta^t(1-\beta)|\nabla_i f(\x^{k-t})|^2}{\sqrt{\beta^{t+1}\v_i^{k-t-1}+\beta^t(1-\beta)|\nabla_i f(\x^{k-t})|^2+(1-\beta^{t+1})\sigma_i^2}}\right]\\
\leq&\E_{\F_{k-t-1}}\left[\sqrt{\beta^{t+1}\v_i^{k-t-1}+(1-\beta^{t+1})\sigma_i^2}\right]\\
&+\E_{\F_{k-t-1}}\left[\frac{\beta^t(1-\beta)|\nabla_i f(\x^{k-t})|^2}{\sqrt{\beta^{t+1}\v_i^{k-t-1}+\beta^t(1-\beta)|\nabla_i f(\x^{k-t})|^2+(\beta^t-\beta^{t+1})\sigma_i^2}}\right]\\
=&\E_{\F_{k-t-1}}\left[\sqrt{\beta^{t+1}\v_i^{k-t-1}+(1-\beta^{t+1})\sigma_i^2}\right]+\sqrt{\beta^t}(1-\beta)\E_{\F_{k-t-1}}\left[\frac{|\nabla_i f(\x^{k-t})|^2}{\sqrt{\widetilde\v_i^{k-t}}}\right],
\end{aligned}
\end{eqnarray}
where we use the concavity of $\sqrt{x}$ in $\overset{(1)}\leq$ and Assumptions 2 and 3 in $\overset{(2)}\leq$. Applying the above inequality recursively for $t=1,2,\cdots,k-1$, we have
\begin{eqnarray}
\begin{aligned}\notag
&\E_{\F_{k-1}}\left[\sqrt{\beta\v_i^{k-1}+(1-\beta)\sigma_i^2}\right]\\
\leq&\sqrt{\beta^k\v_i^0+(1-\beta^k)\sigma_i^2}+\sum_{t=1}^{k-1}\sqrt{\beta^{k-t}}(1-\beta)\E_{\F_{t-1}}\left[\frac{|\nabla_i f(\x^t)|^2}{\sqrt{\widetilde\v_i^t}}\right]
\end{aligned}
\end{eqnarray}
and
\begin{eqnarray}
\begin{aligned}\notag
\E_{\F_{k-1}}\left[\sqrt{\widetilde\v_i^k}\right]\leq& \sqrt{\beta^k\v_i^0+(1-\beta^k)\sigma_i^2}+\sum_{t=1}^k\sqrt{\beta^{k-t}}(1-\beta)\E_{\F_{t-1}}\left[\frac{|\nabla_i f(\x^t)|^2}{\sqrt{\widetilde\v_i^t}}\right]\\
\leq&\sqrt{\sigma_i^2+\frac{1}{dT}}+\sum_{t=1}^k\sqrt{\beta^{k-t}}(1-\beta)\E_{\F_{t-1}}\left[\frac{|\nabla_i f(\x^t)|^2}{\sqrt{\widetilde\v_i^t}}\right],
\end{aligned}
\end{eqnarray}
where we use $\v_i^0=\lambda\max\left\{\sigma_i^2,\frac{1}{dT}\right\}\leq \sigma_i^2+\frac{1}{dT}$. Summing over $i=1,2,\cdots,d$ and $k=1,2,\cdots,K$, we have
\begin{eqnarray}
\begin{aligned}\notag
\sum_{i=1}^d\sum_{k=1}^K \E_{\F_{k-1}}\left[\sqrt{\widetilde\v_i^k}\right]\leq& K\sum_{i=1}^d\left(\sigma_i+\frac{1}{\sqrt{dT}}\right)+\sum_{k=1}^K\sum_{t=1}^k\sqrt{\beta^{k-t}}(1-\beta)\sum_{i=1}^d\E_{\F_{t-1}}\left[\frac{|\nabla_i f(\x^t)|^2}{\sqrt{\widetilde\v_i^t}}\right]\\
=& K\|\bsigma\|_1+\sqrt{dT}+\sum_{t=1}^K\sum_{k=t}^K\sqrt{\beta^{k-t}}(1-\beta)\sum_{i=1}^d\E_{\F_{t-1}}\left[\frac{|\nabla_i f(\x^t)|^2}{\sqrt{\widetilde\v_i^t}}\right]\\
\leq& K\|\bsigma\|_1+\sqrt{dT}+\frac{1-\beta}{1-\sqrt{\beta}}\sum_{t=1}^K\sum_{i=1}^d\E_{\F_{t-1}}\left[\frac{|\nabla_i f(\x^t)|^2}{\sqrt{\widetilde\v_i^t}}\right]\\
=& K\|\bsigma\|_1+\sqrt{dT}+(1+\sqrt{\beta})\sum_{t=1}^K\sum_{i=1}^d\E_{\F_{t-1}}\left[\frac{|\nabla_i f(\x^t)|^2}{\sqrt{\widetilde\v_i^t}}\right].
\end{aligned}
\end{eqnarray}
\end{proof}

\section{Experimental Details}

We conduct experiments on both computer vision and natural language processing tasks to verify the relationship $\|\nabla f(\x^k)\|_1=\varTheta(\sqrt{d})\|\nabla f(\x^k)\|_2$ during the training of RMSProp and its momentum variant. We call the \texttt{torch.optim.RMSprop} API in PyTorch for both optimizers. Code is released at 
\begingroup
\urlstyle{rm}
\small\url{https://github.com/adonis-dym/Convergence-Rate-RMSProp}
\endgroup.

For computer vision experiments, we train ResNet-50 on both CIFAR-100 and ImageNet datasets. In the CIFAR-100 training task, we set the initial learning rate to $10^{-5}$ and employ a cosine decay scheduler over all 100 training epochs. We set the batch size to 64 and the weight decay to 0.1. For the ImageNet training task, we utilize the timm library protocol \citep{rw2019timm}. The training process spans 200 epochs with a 20-epoch linear warm-up period to increase the learning rate to $10^{-4}$, a 170-epoch cosine decay period to decreases the learning rate to $10^{-5}$, and 10 final epochs with a constant learning rate $10^{-5}$. We set the batchsize to 2048. We maintain identical settings for both optimizers and configure the other parameters using the default settings in the PyTorch API, including assigning the momentum parameter to 0.9 for RMSProp with momentum. At the end of each epoch, we compute the full training loss and gradient by traversing the entire training dataset to accurately measure the gradient norm ratio $\frac{\|\nabla f(\x^k)\|_1}{\|\nabla f(\x^k)\|_2}$.

For natural language processing tasks, we train the classic GPT-2 model from scratch on the OpenWebText dataset using the Megatron-LM framework. Setting the batchsize to 640, we train the model for 50000 steps, equivalent to approximately 3.5 epochs. The training schedule includes a 2000-step linear warm-up period increasing the learning rate to $10^{-5}$ and a cosine decay period for the remaining steps. In our training setting, we employ a decoupled weight decay of 0.05 in the AdamW style, rather than the vanilla implementation of $\ell_2$ regularization in the PyTorch API. Given the computational constraints inherent in large-scale language model training, we approximate full gradients by aggregating over a subset of 100 batches, providing an efficient yet representative estimate of the full gradients.

Our experimental results, as compiled in Figure \ref{figure1} in Section \ref{sec:intr}, demonstrate that the gradient norm ratio $\frac{\|\nabla f(\x^k)\|_1}{\|\nabla f(\x^k)\|_2}$ consistently scales as $\varTheta(\sqrt{d})$. This empirical observation confirms that the convergence rate derived in this study is in accordance with that of SGD with respect to the problem dimension $d$.

\section*{Conclusion and Future Work}

This paper studies the classical RMSProp and its momentum extension. We establish the convergence rate of $\frac{1}{T}\sum_{k=1}^T\E\left[\|\nabla f(\x^k)\|_1\right]\leq \widetilde\bO\left(\frac{\sqrt{d}}{T^{1/4}}\left(\sqrt[4]{\sigma_s^2L(f(\x^1)-f^*)}\right)\right)$ measured by $\ell_1$ norm without the bounded gradient condition. Our convergence rate can be considered to be analogous to the $\frac{1}{T}\sum_{k=1}^T\E\left[\|\nabla f(\x^k)\|_2\right]\leq \bO\left(\frac{1}{T^{1/4}}\sqrt[4]{\sigma_s^2L(f(\x^1)-f^*)}\right)$ one of SGD in the ideal case of $\|\nabla f(\x)\|_1=\varTheta(\sqrt{d})\|\nabla f(\x)\|_2$ for high-dimensional problems. One interesting future work is to establish the lower bound of adaptive gradient methods measured by $\ell_1$ norm. We conjecture that the lower bound is $\bO\left(\frac{\sqrt{d}}{T^{1/4}}\left(\sqrt[4]{\sigma_s^2L(f(\x^1)-f^*)}\right)\right)$.

\appendix
\section{Proof of Lemma \ref{lemma1}}

\begin{proof}
From $\ln(1-x)\leq -x$ for any $x<1$, we have
\begin{eqnarray}
\begin{aligned}\notag
&(1-\beta)\frac{g_t^2}{v_t}\leq-\ln\left(1-(1-\beta)\frac{g_t^2}{v_t}\right)=-\ln\frac{v_t-(1-\beta)g_t^2}{v_t}=-\ln\frac{\beta v_{t-1}}{v_t}=\ln\frac{v_t}{\beta v_{t-1}}
\end{aligned}
\end{eqnarray}
and
\begin{eqnarray}
\begin{aligned}\notag
&(1-\beta)\sum_{t=1}^k\frac{g_t^2}{v_t}\leq\ln\frac{v_k}{\beta^kv_0}.
\end{aligned}
\end{eqnarray}
\end{proof}

\section{Proof of $c\ln x\leq x$ for all $x\geq 3c\ln c$ and $c\geq 3$}\label{appendixB}
\begin{proof}
Denote $f(x)=\frac{\ln x}{x}$. Since $f'(x)=\frac{1}{x^2}-\frac{\ln x}{x^2}\leq 0$ when $x\geq e$, $f(x)$ is decreasing when $x\geq e$ and $\frac{\ln x}{x}\leq\frac{\ln (3c\ln c)}{3c\ln c}=\frac{1}{c}\frac{\ln c+\ln(3\ln c)}{3\ln c}\leq \frac{1}{c}\frac{\ln c+\ln c^2}{3\ln c}=\frac{1}{c}$, where we use $\ln c\leq c$ and $3\ln c\leq c^2$ for $c\geq 3$.
\end{proof}

\bibliography{RMSProp}
\bibliographystyle{icml2020}

\end{document}